\documentclass[11pt]{article}
\usepackage{amsmath,amssymb,amsthm,mathtools}
\usepackage{geometry}
\usepackage{enumitem}
\usepackage{microtype}
\usepackage{enumitem} 
\usepackage{amsmath}
\usepackage[hidelinks]{hyperref}
\hypersetup{
  pdftitle={
    Counterexamples to Whole-Sequence Convergence of
    Variable-Smoothing Full-Splitting Methods
  },
  pdfauthor={Min Tao},
  pdfkeywords={
    structured fractional programming;
    variable smoothing;
    whole-sequence convergence;
    Whitney extension criterion;
    limiting lifted stationarity;
    counterexamples
  }
}

\DeclareMathOperator{\sinc}{sinc}
\numberwithin{equation}{section}
\newtheorem{theorem}{Theorem}[section]
\newtheorem{lemma}[theorem]{Lemma}
\newtheorem{proposition}[theorem]{Proposition}
\newtheorem{assumption}[theorem]{Assumption}

\newtheorem{remark}[theorem]{Remark}

\newtheorem{algorithm}[theorem]{Algorithm}
\newtheorem{definition}[theorem]{Definition}

\newtheorem{corollary}[theorem]{Corollary}
\newcommand{\R}{\mathbb{R}}

\newcommand{\ip}[2]{\left\langle #1,#2\right\rangle}
\newcommand{\norm}[1]{\left\lVert #1\right\rVert}
\newcommand{\Lip}{\operatorname{Lip}}

\newcommand{\Sph}{\mathbb S}
\newcommand{\Sone}{\mathbb S^1}

\newcommand{\clust}{\operatorname{cluster}}

\newcommand{\dist}{\operatorname{dist}}
\newcommand{\prox}{\operatorname{prox}}
\newcommand{\Proj}{\operatorname{Proj}}

\newcommand{\nn}{\nonumber}

\DeclareMathOperator{\rank}{rank}
\DeclareMathOperator{\dom}{dom}

\title{
Counterexamples to Whole-Sequence Convergence of
Variable-Smoothing Full-Splitting Methods
}

\author{
Min Tao\\[1mm]
\small School of Mathematics, Nanjing University\\
\small National Key Laboratory for Novel Software Technology\\
\small Nanjing 210093, China\\
\small \href{mailto:taom@nju.edu.cn}{taom@nju.edu.cn}
}

\date{}
\begin{document}

\maketitle

\begin{abstract}
We study  whole-sequence convergence of the smoothing-based
full-splitting proximal subgradient method (S-FSPS) for structured nonconvex
and nonsmooth fractional programs, introduced  by
Bo\c{t}, Li, and Tao (SIAM J. Optim., 35(4):2623--2653, 2025) as Algorithm~4.1. Existing theory
guarantees only the existence of a subsequence converging to a
limiting lifted stationary point. We show that this guarantee is sharp by constructing two admissible
instances whose corresponding primal sequences both have cluster set
$\{1\}\times\mathbb S^1$ and infinite length, although every cluster
point is a limiting lifted stationary point.  The construction prescribes a slowly rotating spiral and
realizes it exactly through a compatible first-order jet and a
$C^{1,1}$ Whitney extension. In the first instance, $A$ has rank one
and every smoothing-dual iterate is nonzero. In the second, the
feasible set is full-dimensional, $A$ has full row rank, and each of
$f\circ K$, $g\circ A$, and the numerator $g\circ A+h$ is
nonconstant on the feasible set. The first instance also yields a nonconvergent example for the
corresponding variable-smoothing, single-loop, full-splitting method
for nonconvex and nonsmooth composite optimization, although that
method still admits a subsequence converging to an exact stationary
point. Thus, a vanishing but
nonsummable smoothing schedule does not imply whole-sequence
convergence.
\end{abstract}
\vspace{0.2cm}


\section{Introduction}\label{sec:introduction}

Whole-sequence convergence is substantially stronger than subsequential
convergence to stationary points. In particular, boundedness of the generated
sequence and stationarity of all its cluster points do not imply convergence
to a single limit. A standard way to upgrade subsequential
convergence to whole-sequence convergence is provided by the
Kurdyka--\L ojasiewicz (KL) framework \cite{AB,ABS,Sun1}. When a suitable KL
inequality is combined with a sufficient-decrease condition, a relative-error
estimate, and appropriate control of perturbation terms, one can often
establish convergence of the entire sequence; see, for example,
\cite{AB,ABS,Sun1}. The distinction between subsequential stationarity and
whole-sequence convergence is central to the present work.


Smoothing methods for nonsmooth optimization have traditionally been analyzed
in terms of objective-value accuracy, iteration complexity, or subsequential
stationarity. Classical contributions include Nesterov's smoothing technique
\cite{Nesterov2005}, the unified framework of Beck and Teboulle
\cite{BeckTeboulle2012}, and Chen's gradient-consistent approximation framework
for nonconvex problems \cite{Chen2012}. Variable-smoothing schemes were later
developed for convex composite optimization
\cite{BotBohm2020,BotHendrich2015} and for nonconvex composite models
\cite{BohmWright2021,LiuXia2024,LongEtAl2026}. Kume and Yamada~\cite{KumeYamada2025} analyze a single-loop
variable-smoothing method for weakly convex composite minimization
with a manifold constraint via parametrization, obtaining
subsequential stationarity rather than whole-sequence convergence.
 These results provide
complexity guarantees for achieving $\varepsilon$-approximate stationarity or
accumulation-point guarantees for exact limiting (lifted) stationary points,
but do not generally establish whole-sequence convergence to an exact
stationary point.

Whole-sequence convergence is available in more specialized settings. Sabach,
Teboulle, and Voldman \cite{SabachTeboulleVoldman2018} establish finite length
and convergence for a clustering model by combining the KL property with a
fixed smoothing parameter.  Bian and Chen~\cite{BianChen2020} establish whole-sequence convergence,
without invoking the KL property, for an exact capped-$\ell_1$ relaxation
of a cardinality-penalized sparse regression model. Their proof exploits
the model-specific lower-bound and support-identification structure of the
capped-$\ell_1$ penalty, together with an adaptive smoothing update satisfying
\(\sum_{k=0}^{\infty}\mu_k^2<+\infty\). However, square summability of the smoothing parameters
alone is insufficient for whole-sequence convergence, as the counterexamples
constructed below demonstrate.


For the structured fractional optimization problem (\ref{eq:fractional-problem}),
Bo\c{t}, Li, and Tao
\cite{BotLiTao2025} introduced the smoothing-based full-splitting proximal
subgradient method (S-FSPS).  The method separates the constituent functions and linear operators
and employs a nonsummable smoothing-parameter sequence that decreases
to zero.
The convergence theory in \cite{BotLiTao2025} guarantees, under the standing
assumptions, the existence of subsequences converging to exact limiting lifted
stationary points.  It does not establish convergence of the entire
sequence.  This leaves open whether
whole-sequence convergence of S-FSPS follows under its standing assumptions.

The KL framework \cite{AB,ABS,Sun1} does not directly resolve this
issue. Standard KL-based convergence analyses typically establish
the finite-length property by combining a KL inequality for a fixed
or suitably augmented Lyapunov function with a sufficient-decrease
estimate, a relative-error bound with uniform coefficients, and
summable perturbation terms. These ingredients are not guaranteed
for S-FSPS under its standing assumptions. In particular, the
assumptions impose neither the KL property nor a structural
condition, such as definability, that would ensure it.

Even if an appropriate KL property is additionally assumed, the
currently available smoothing-dependent estimates do not fit the
standard finite-length framework. As the smoothing parameter tends
to zero, the coefficients in the relative-error estimate become
nonuniform, while some residual terms are neither controlled by the
available descent inequality nor known to be summable. Treating the
smoothing parameter as an additional variable in an augmented
Lyapunov function does not remove these difficulties under the
currently available estimates. Moreover, the prescribed
nonsummable smoothing schedule is incompatible with stronger
summability conditions that arise in existing inexact KL analyses
\cite{Sun1}. Consequently, the available KL theory cannot be applied
directly to establish whole-sequence convergence of S-FSPS under its
standing assumptions.

These observations motivate us to investigate whether the primal sequence generated by S-FSPS converges. More generally, we ask whether such convergence
can be derived for variable-smoothing, single-loop, full-splitting algorithms
for broader classes of structured optimization problems.
Nonconvergence of iterates can occur even in smooth convex optimization. Bolte
and Pauwels \cite{BoltePauwels2021} give examples in which classical algorithms
and associated dynamical systems fail to converge. Bolte, Combettes, and
Pauwels \cite{BolteCombettesPauwels2024} show that Frank--Wolfe iterates may
fail to converge over a compact convex set even when their objective values
converge to the optimum. Our setting is different: the nonconvergence arises
in a variable-smoothing full-splitting method for structured fractional
optimization.

We construct a slowly rotating spiral whose angular increments tend to zero
but have a divergent sum. The orbit has the unit circle as its cluster set and
has infinite length. A compatible first-order jet and a
\(C^{1,1}\) Whitney extension are then used to construct problem data for which
S-FSPS generates this orbit exactly. Both instances satisfy the assumptions of
the subsequential convergence theory, yet their primal sequences share the cluster
set \(\{1\}\times\mathbb S^1\) and do not converge, although every cluster
point is an exact limiting lifted stationary point.

\paragraph{Contributions}
The principal contributions are summarized below.
\begin{enumerate}[label=\textup{(\roman*)},leftmargin=2.3em]

\item \textbf{Sharpness of the S-FSPS convergence theory.}
We first refine the subsequential convergence analysis in
\cite[Theorem~4.3]{BotLiTao2025}. Specifically, the explicit bound
$
\operatorname{dist}\bigl(0,\partial g(Ax)\bigr)\leq \ell
\; (x\in\mathcal S)
$
already provides the uniform Moreau-envelope estimates required to establish
the existence of a subsequence converging to an exact limiting lifted
stationary point. Consequently, this subsequential conclusion remains valid
without the  condition
$
A(\mathcal S)\subseteq\operatorname{int}(\operatorname{dom}g).
$
We then show that this subsequence convergence result  cannot be strengthened to
whole-sequence convergence under the standing assumptions. To show this, we
construct an admissible instance whose primal sequence has a continuum of cluster points and an
infinite-length trajectory.

\item \textbf{Exact realization and robustness of the counterexample.}
We prescribe a slowly rotating spiral together with compatible function
values and first-order data, verify the Glaeser--Whitney compatibility
conditions globally on the union of the spiral and its limiting circle, and
invoke a \(C^{1,1}\) Whitney extension theorem
\cite{DaniilidisEtAl2018} to construct a globally defined smooth component
whose gradient has a controlled Lipschitz constant. We then
choose the remaining problem data so that the S-FSPS iteration reproduces the
prescribed orbit exactly. In the basic construction, the linear operator \(A\) is nonzero and
\(z^k\neq 0\) for every \(k\geq 0\).
 We further construct a
full-dimensional fractional instance for which \(A\) has full row rank,
while each of  \(f\circ K\), \(g\circ A\) and $g\circ A+h$ is nonconstant on the
feasible set. Thus, the failure of whole-sequence convergence cannot be
attributed to a lower-dimensional feasible set, degeneracy of the composite
operator, or constancy of the denominator \(f\circ K\), the composite term \(g\circ A\) or the numerator $g\circ A+h$.

\item \textbf{Implications for composite optimization.}
The first construction also yields a nonconvergent example for the corresponding
variable-smoothing, single-loop, full-splitting method applied to
$
\min_{x\in\mathcal S}\; g(Ax)+h(x).
$
The generated sequence still admits a subsequence converging to an exact
 stationary point, while the full sequence fails to converge.

\end{enumerate}

\paragraph{Organization}
The remainder of the paper is organized as follows.
Section~\ref{sec:algorithm} recalls S-FSPS, refines its subsequential
convergence theory, and states the whole-sequence convergence question.
Section~\ref{sec:whitney} presents the Whitney extension tool used in the
construction. Section~\ref{sec:construction} develops the basic and
full-dimensional fractional counterexamples. Section~\ref{sec:com}
treats the corresponding composite optimization method. Finally,
Section~\ref{sec:conclusion} concludes the paper and discusses the implications
for whole-sequence convergence.

\section{S-FSPS and the Whole-Sequence Convergence Question}
\label{sec:algorithm}
 The Euclidean norm is denoted by \(\|\cdot\|\), and
\(\langle\cdot,\cdot\rangle\) denotes the corresponding inner product.
Let $\overline \R:=\R\cup\{+\infty\}$.
For a function \(f:\mathbb{R}^n\to\overline{\mathbb{R}}\), its
\emph{effective domain} is defined by
$
\dom f:=\{x\in\mathbb{R}^n:f(x)<+\infty\},
$
and \(f\) is said to be \emph{proper} if \(\dom f\neq\emptyset\).
The \emph{(Fenchel) conjugate} of \(f\) is given by
$
f^*(v):=\sup_{x\in\mathbb{R}^n}
\{\langle v,x\rangle-f(x)\}.
$
The domain of the convex subdifferential is denoted by
$
\dom(\partial f):=\{x\in\mathbb{R}^n:\partial f(x)\neq\emptyset\}.
$
For a proper, convex, and lower semicontinuous function
\(f:\mathbb{R}^n\to\overline{\mathbb{R}}\), its \emph{proximal operator}
with parameter \(\gamma>0\) is defined by
$
\prox_{\gamma f}(x)
:=
\arg\min_{y\in\mathbb{R}^n}
\left\{
f(y)+\frac{1}{2\gamma}\|y-x\|^2
\right\}.
$
The corresponding \emph{Moreau envelope} is defined by
$
f_\gamma(x)
:=
\min_{y\in\mathbb{R}^n}
\left\{
f(y)+\frac{1}{2\gamma}\|y-x\|^2
\right\}.
$
For every \(x\in\mathbb{R}^n\),
$
f_\gamma(x)
=
\left(f^*+\frac{\gamma}{2}\|\cdot\|^2\right)^*(x),
$
and \(f_\gamma\) is continuously differentiable with
$
\nabla f_\gamma(x)
=
\frac{x-\prox_{\gamma f}(x)}{\gamma}
=
\prox_{f^*/\gamma}\left(\frac{x}{\gamma}\right).
$
Moreover,
$
\nabla f_\gamma(x)
\in
\partial f\bigl(\prox_{\gamma f}(x)\bigr).
$
We use
$
    \Sone:=\bigl\{x\in\mathbb R^2:\|x\|=1\bigr\}
    \;\text{and}\;
    \overline{\mathbb B}_{\mathbb R^2}
    :=\bigl\{x\in\mathbb R^2:\|x\|\leq1\bigr\}
$
to denote the unit circle and the closed unit ball in \(\mathbb R^2\),
respectively. When the  Euclidean space is clear from the context,
we  write \(\overline{\mathbb B}\) for the corresponding closed unit
ball centered at the origin.
For a nonempty closed convex set
\(\mathcal C\subseteq\mathbb{R}^n\), its \emph{indicator function} is defined by
\[
\iota_{\mathcal C}(x)
:=
\begin{cases}
0, & \text{if }x\in\mathcal C,\\
+\infty, & \text{otherwise}.
\end{cases}
\]
The \emph{normal cone} to \(\mathcal C\) at \(x\) is denoted by
$
N_{\mathcal C}(x):=\partial\iota_{\mathcal C}(x)
$, and
$\Proj_{\mathcal C}(\cdot)$ denotes the Euclidean projection onto $\mathcal C$.
For a sequence \(\{x^k\}_{k\ge0}\), we denote by
$
\operatorname{cluster}(x^k)
$
the set of its cluster points.
Given a linear operator $A: {\mathbb R}^n \rightarrow {\mathbb R}^m$, we denote by $A^*: {\mathbb R}^m \rightarrow {\mathbb R}^n$ its {\it adjoint operator}. We also use $\sigma_A:=\|A\|=\sup\{ \|A  x\|: \| x\|=1\}$ to denote its norm.
Given a vector \(x\in\mathbb{R}^n\), \(x_i\) denotes its \(i\)th component.


In this paper, we consider the structured fractional optimization problem
\begin{equation}\label{eq:fractional-problem}
    \min_{x\in\mathcal S}
    \frac{g(Ax)+h(x)}{f(Kx)},
\end{equation}
where \(\mathcal S\) is a nonempty, convex, and compact subset of
\(\mathbb R^n\);
\(f:\mathbb R^p\to\overline{\mathbb R}:=\mathbb R\cup\{+\infty\}\) and
\(g:\mathbb R^s\to\overline{\mathbb R}\) are proper, convex, and lower
semicontinuous;
\(A:\mathbb R^n\to\mathbb R^s\) and
\(K:\mathbb R^n\to\mathbb R^p\) are linear operators; and
\(h:\mathbb R^n\to\mathbb R\) is  differentiable on an open set containing ${\cal S}$  and its gradient $\nabla h$ is $L_{\nabla h}$-Lipschitz continuous over this  set.


We use the following  assumption of \cite[Assumption~3.1]{BotLiTao2025}.
\begin{assumption}\label{ass:standing}
The set \(\mathcal S\subseteq\mathbb R^n\) is nonempty, convex, and compact;
\(A\) and \(K\) are nonzero linear operators; \(g\) and \(f\) are proper,
convex, and lower semicontinuous functions; and \(h\) is differentiable with an
\(L_{\nabla h}\)-Lipschitz continuous gradient on an open neighborhood of
\(\mathcal S\). Moreover,
\begin{enumerate}[label=\textup{(\roman*)}]
    \item \(K(\mathcal S)\subseteq\operatorname{int}(\dom f)\) and
          \(f(Kx)>0\) for every \(x\in\mathcal S\);
    \item \(\mathcal S\cap A^{-1}(\dom g)\neq\varnothing\) and
          \(\inf_{x\in\mathcal S}\{g(Ax)+h(x)\}>0\);
    \item \(A(\mathcal S)\subseteq\dom(\partial g)\), and there exists
          \(\ell>0\) such that
          \(\dist(0,\partial g(Ax))\le\ell\) for every
          \(x\in\mathcal S\).
\end{enumerate}
\end{assumption}
\begin{remark}
The condition
$
\mathcal S\cap A^{-1}(\dom g)\!\neq\!\varnothing
$ in Assumption~\ref{ass:standing}(ii)
is redundant under Assumption~\ref{ass:standing}(iii). Since
$
A(\mathcal S)\subseteq\dom(\partial g)\subseteq\dom g,
$
we have
$
\mathcal S\subseteq A^{-1}(\dom g),
$
which immediately yields
$
\mathcal S\cap A^{-1}(\dom g)=\mathcal S\neq\varnothing.
$
Moreover, if \(g:\R^s\rightarrow\R\) is globally Lipschitz continuous,
Assumption~\ref{ass:standing}(iii) is automatically satisfied. Indeed, if
\(L:=\Lip(g)\), then
$\partial g(u)\subseteq L\,\overline{\mathbb B}$ for all $u\in\R^s$.
\end{remark}
We review the concept of limiting  lifted stationary point of (\ref{eq:fractional-problem}).

\begin{definition}\label{def:limiting-lifted-stationarity}
For problem~\eqref{eq:fractional-problem}, a point
\[
x\in\mathcal S\cap A^{-1}(\dom\partial g)
\cap K^{-1}(\dom\partial f)
\]
with \(f(Kx)>0\) is called a \emph{limiting lifted stationary point} if
\[
0\in
f(Kx)
\bigl(
A^*\partial g(Ax)+\nabla h(x)+\partial\iota_{\mathcal S}(x)
\bigr)
-
(g(Ax)+h(x))K^*\partial f(Kx).
\]
\end{definition}


\noindent Next, we define the smoothed potential function with $\gamma>0$,
\begin{equation}\label{eq:Psi}
\Psi(x,z;\gamma)
:=
\langle z,Ax\rangle-g^*(z)+h(x)+\iota_{\mathcal S}(x)
-\frac{\gamma}{2}\|z\|^2 .
\end{equation}


The following lemma shows that the explicit uniform subgradient bound in Assumption~\ref{ass:standing}(iii) provides all the local control needed in the proof of Theorem~\ref{PriTheo2R}.
\begin{lemma}
\label{lem:uniform-moreau-estimates}
Suppose that $g$ is proper, convex, and lower semicontinuous and that Assumption \ref{ass:standing}(iii) holds.
Then, for every $w\in A(\mathcal S)$ and every $\gamma>0$,
\begin{itemize}
\item[(i)]
$\|\nabla g_\gamma(w)\|\leq\ell$;
\label{eq:uniform-gradient-bound}
\item[(ii)]
$
\|\prox_{\gamma g}(w)-w\|\leq\gamma\ell$;
\label{eq:uniform-proximal-bound}
\item[(iii)]
$0\leq g(w)-g_\gamma(w)\leq\frac{\gamma\ell^2}{2}$.

\item[(iv)]
$|g(u)-g(v)|\leq\ell\|u-v\|
\;\;\text{for any }u,v\in A(\mathcal S).$
\end{itemize}
\end{lemma}
The proof of Lemma \ref{lem:uniform-moreau-estimates} is deferred to Appendix \ref{appA}.

\noindent We next recall the S-FSPS algorithm proposed in  \cite[Algorithm~4.1]{BotLiTao2025}.
\begin{algorithm}[S-FSPS]\label{alg:sfsps}
Let $\{\gamma_k\}_{k\ge0}$ be positive and nonincreasing, with
\begin{eqnarray}\label{eq:gammak}
\gamma_k\to0,
\qquad
\sum_{k=0}^{\infty}\gamma_k=+\infty.
\end{eqnarray}
Choose $\chi>1$, $\theta_0>0$, and a positive upper bound $L_{\nabla h}$ on the Lipschitz constant of
$\nabla h$ over the prescribed neighborhood, and
\begin{equation}\label{eq:delta-general}
\delta_k:=\chi\left(L_{\nabla h}+\frac{\sigma_A^2}{\gamma_k}\right).
\end{equation}
Starting from $(x^0,z^0)$ with $x^0\in\mathcal S$, generate, for $k\ge0$,
\begin{align}
 y^{k+1}&\in\partial f(Kx^k),\label{eq:y-update}\\
 x^{k+1}&:=\Proj_{\mathcal S}\left(
 x^k+\frac{\theta_k}{\delta_k}K^*y^{k+1}
 -\frac{1}{\delta_k}\nabla h(x^k)
 -\frac{1}{\delta_k}A^*z^k\right),\label{eq:x-update}\\
 z^{k+1}&:=\prox_{g^*/\gamma_k}\left(\frac{Ax^{k+1}}{\gamma_k}\right),\label{eq:z-update}\\
 \theta_{k+1}&:=\frac{\Psi(x^{k+1},z^{k+1};\gamma_k)}{f(Kx^{k+1})}.
 \label{eq:theta-update}
\end{align}
\end{algorithm}
We next refine the convergence analysis of
Algorithm~\ref{alg:sfsps} presented in \cite{BotLiTao2025}.
Specifically, we show that conclusions~(vi)--(vii) of
\cite[Theorem~4.3]{BotLiTao2025} remain valid without the assumption
$
    A(\mathcal{S})\subseteq\operatorname{int}(\dom g).$
\begin{theorem}
\label{PriTheo2R}
Suppose Assumption \ref{ass:standing} holds. Define $V^k := (x^k,y^k,z^k), \; k\geq 1$. Let $\Omega$ be the set of  cluster points of the sequence $\{{ V}^k\}$ generated by Algorithm \ref{alg:sfsps}.  Then the following statements hold:
\begin{itemize}
\item[(i)] For every $k \geq 1$, the following inequality holds:
\begin{align*}\label{desineqRR}
& \ \Psi({ x}^{k+1}, { z}^{k+1}; \gamma_k)
+ \theta_k \left[f(K{ x}^k)-\left( \langle K{ x}^{k+1},{ y}^{k+1} \rangle - f^*({ y}^{k+1})\right)\right]\nn \\
\leq & \ \Psi({ x}^{k}, { z}^{k}; \gamma_{k-1})
-{\widetilde c}_{k}\|{ x}^k-{ x}^{k+1}\|^2 + \Xi^{k+1},
\end{align*}
 where
$
\Xi^{k+1} := \frac{\gamma_{k-1}-\gamma_{k}}{2}\|{ z}^{k+1}\|^2 \geq 0 \quad \mbox{and} \quad
 {\widetilde c}_{k}:=\frac{(\chi-1)}{2}\left(L_{\nabla h} +\frac{\sigma_A^2}{\gamma_{k}}\right) >0.
$

\item[(ii)] The sequence $\{{ V}^k\}$ is bounded.
\item[(iii)] There exists an index $K_1 \geq 1$ such that $\theta_k\ge 0$ for all $k\ge K_1$.
\item[(iv)] $\lim_{k\to+\infty}\theta_k=\overline{\theta}$ {for some $\overline{\theta} \ge 0$.}
\item[(v)] We have $\liminf\limits_{k\to+\infty}\delta_{k}\|{ x}^{k+1}-{ x}^k\|=0.$
\end{itemize}
\begin{itemize}
\item[(vi)]
{For every $(\overline{ x},{\overline{ y}},{\overline{ z}}) \in \Omega$, we have }
$\frac{g(A{\overline{ x}})+h({\overline{ x}})+\iota_{\cal S}(\overline{ x})}{f(K{\overline{ x}})}={\overline{ \theta}},$
{where $\overline{\theta}$ is given as in {(iv)}.}
\item[(vii)]
Let \(\{{ x}^{k_j}\}\) be a subsequence of $\{x^k\}_{k\ge 0}$ such that
$
\lim\limits_{j \to +\infty} \delta_{k_j} \|{ x}^{k_j+1} - { x}^{k_j}\| = 0$ (whose existence is guaranteed by {(v)}).  Then every cluster point \(\overline{ x}\) of this subsequence is a limiting lifted stationary point for the optimization problem \eqref{eq:fractional-problem}.
\end{itemize}
\end{theorem}
The proof of Theorem~\ref{PriTheo2R} is deferred to Appendix~\ref{appB}.

\medskip
\noindent
Theorem~\ref{PriTheo2R} guarantees the existence of a subsequence converging to an exact limiting lifted stationary point. However, it remains {\it open} whether the entire primal sequence generated by S-FSPS necessarily converges.
Section \ref{sec:construction} answers this question in the negative.

\section{Whitney Extension Criterion}\label{sec:whitney}
We first recall a Glaeser-Whitney \(C^{1,1}\) extension theorem \cite{DaniilidisEtAl2018} and then derive a practical extension criterion.
Let $S$ be a nonempty subset of a Hilbert space
$(\mathcal H,\langle\cdot,\cdot\rangle,\|\cdot\|)$ and assume
$\alpha:S\rightarrow \mathbb{R}$ and $v:S\rightarrow \mathcal H$ satisfy
the so-called Glaeser--Whitney conditions:
\begin{equation}\label{eq:GW}
\begin{aligned}
&K_1:=\sup_{\substack{s_1,s_2\in S\\s_1\neq s_2}}
\frac{\left|
\alpha(s_2)-\alpha(s_1)
-\langle v(s_1),s_2-s_1\rangle
\right|}
{\|s_1-s_2\|^2}
<+\infty,\\[2ex]
&K_2:=\sup_{\substack{s_1,s_2\in S\\s_1\neq s_2}}
\frac{\|v(s_1)-v(s_2)\|}
{\|s_1-s_2\|}
<+\infty .
\end{aligned}
\end{equation}
\begin{theorem}\label{lem:whitney}(
  \(C^{1,1}\)-Glaeser--Whitney almost-minimal extension;
  see \cite[Theorem~3.1]{DaniilidisEtAl2018}
)
Let $S$ be a nonempty subset of a Hilbert space $H$ and let
$(\alpha(s),v(s))_{s\in S}$ be a $1$-Taylor field on $S$ satisfying
(\ref{eq:GW}). Then, the function
\[
G(x)=F(x)-\frac{1}{2}\overline{\mu}\|x\|^2
\]
is an explicit $C^{1,1}$-extension of the $1$-Taylor field $(\alpha,v)$,
provided that $F$ is the convex extension of the $1$-Taylor field
$(\tilde{\alpha},\tilde v)$ where for all $s\in S$,
$
\tilde{\alpha}(s):=\alpha(s)+\frac{1}{2}\overline{\mu}\|s\|^2
$
and
$
\tilde v(s):=v(s)+\overline{\mu}s,
$
with
$
\overline{\mu}
=
2K_1+K_2+
\sqrt{(2K_1+K_2)^2+K_2^2},
$
where $K_1,K_2$ are given by (\ref{eq:GW}).
For \(s_1\ne s_2\), define
\begin{align}
A_{s_1s_2}
&:=
\frac{
2\bigl(\alpha(s_1)-\alpha(s_2)\bigr)
+\ip{v(s_1)+v(s_2)}{s_2-s_1}}
{\norm{s_1-s_2}^{2}},\label{Afor}\\
B_{s_1s_2}
&:=
\frac{\norm{v(s_1)-v(s_2)}}{\norm{s_1-s_2}},\label{Bfor}
\end{align}
and set
\[
\Gamma_1(S,(\alpha,v))
:=
\sup_{s_1\ne s_2}
\left(
\sqrt{A_{s_1s_2}^{\,2}+B_{s_1s_2}^{\,2}}
+|A_{s_1s_2}|
\right).
\]
Moreover, the extension $G$, which satisfies
$G(s)=\alpha(s)$ and $\nabla G(s)=v(s)$ for every $s\in S$,
 is almost minimal, i.e.
\[
\Gamma_1(S,(\alpha,v))
\leq
\Gamma_1(H,(G,\nabla G))
=
\operatorname{Lip}(\nabla G)
\leq
\frac{5+\sqrt{29}}{2}
\Gamma_1(S,(\alpha,v)).
\]
\end{theorem}

%

For the counterexample construction, we do not need an explicit formula for
the extension in Theorem~\ref{lem:whitney}. Instead, we shall invoke the
following  \(C^{1,1}\) Whitney extension criterion.

\begin{corollary}[A practical  \(C^{1,1}\) Whitney extension criterion]
\label{cor:practical-whitney}
Let \(E\subseteq\mathbb R^m\) be nonempty and let
\[
a:E\to\mathbb R,
\qquad
v:E\to\mathbb R^m.
\]
Suppose that there exists \(M>0\) such that, for every \(p,q\in E\),
\begin{align}
\|v(p)-v(q)\|
&\le M\|p-q\|,
\label{eq:whitney-gradient-condition}\\
\bigl|a(p)-a(q)-\langle v(q),p-q\rangle\bigr|
&\le M\|p-q\|^2.
\label{eq:whitney-value-condition}
\end{align}
Then there exists a function \(H\in C^{1,1}(\mathbb R^m)\) such that
\[
H|_E=a,
\qquad
\nabla H|_E=v,
\]
and
\[
\operatorname{Lip}(\nabla H)
\le C_{\mathrm W}M,
\]
where
\[
C_{\mathrm W}
:=
\frac{5+\sqrt{29}}{2}\bigl(3+\sqrt{10}\bigr)
<32.
\]
In particular, \(C_{\mathrm W}\) is a universal constant independent of
the dimension \(m\).
\end{corollary}

\begin{proof}
For distinct \(p,q\in E\), let \(A_{p,q}\) and \(B_{p,q}\) be defined as in
\eqref{Afor} and \eqref{Bfor}, respectively. By
(\ref{eq:whitney-gradient-condition}),
$
B_{p,q}\le M.
$
Define
\[
R_{p,q}
:=
a(p)-a(q)-\langle v(q),p-q\rangle.
\]
By \eqref{eq:whitney-value-condition},
$
|R_{p,q}|\le M\|p-q\|^2.
$
Moreover,
\begin{align*}
&2\bigl(a(p)-a(q)\bigr)
 +\langle v(p)+v(q),q-p\rangle\\
&\qquad
=2R_{p,q}-\langle v(p)-v(q),p-q\rangle.
\end{align*}
Consequently,
\[
|A_{p,q}|
\le
\frac{2|R_{p,q}|}{\|p-q\|^2}
+
\frac{\|v(p)-v(q)\|}{\|p-q\|}
\le 3M.
\]
Therefore,
\[
\Gamma_1(E,(a,v))
\le
\sqrt{(3M)^2+M^2}+3M
=
(3+\sqrt{10})M.
\]
The conclusion follows directly from Theorem~\ref{lem:whitney}.
\end{proof}

\section{Counterexamples}\label{sec:construction}
\subsection{Basic Construction}
\begin{theorem}\label{thm:main-counterexample}
There exists an instance of problem~\eqref{eq:fractional-problem}
satisfying Assumption~\ref{ass:standing} such that  ${\rank}(A)=1$ and $z^k\neq0$ for every $k\ge 0$; moreover, the S-FSPS iterates satisfy
\[
\clust(x^k)=\{1\}\times\Sph^1,
\qquad
\sum_{k=0}^{\infty}\norm{x^{k+1}-x^k}=+\infty.
\]
Moreover, every point in the cluster set is a limiting lifted stationary point.
\end{theorem}
\begin{proof}
\textbf{Step 1: problem data.}
Let $n=3$, $s=2$, $p=3$, and
\[
A=\begin{pmatrix}1&0&0\\0&0&0\end{pmatrix},\qquad
{\cal S}:=\{1\}\times \overline{\mathbb B}_{\R^2},\qquad K=I_3.
\]
Set
\[
g(y)=|y_1|+|y_2|,\qquad f(x)\equiv1.
\]
Then \(g^*=\iota_{[-1,1]^2}\) and
$
\prox_{g^*/\gamma}=\Proj_{[-1,1]^2}.$
Choose \(\chi=2\) and \(\gamma_k=1/(k+K_0)\). We shall construct \(h\) so that
\(\operatorname{Lip}(\nabla h)\le 1\). Since \(\sigma_A=1\), we may take \(1\)
as an upper bound for the Lipschitz constant of \(\nabla h\),
 not necessarily its
exact value. Consequently,
\[
\delta_k=2(k+K_0+1),
\qquad
\alpha_k:=\frac{1}{\delta_k}.
\]


Before stating the main construction, we briefly describe the
auxiliary sequences. The sequence \(\{d_k\}_{k\ge 0}\) measures the radial
distance of the prescribed iterates from the limiting unit circle and is chosen
to vanish as \(k\to\infty\). The sequence \(\{\beta_k\}_{k\ge 0}\) specifies the angular
increment between successive iterates; its nonsummability ensures that the
trajectory continues to rotate indefinitely. These two sequences determine the spiral
\(\{p^k\}_{k\ge 0}\). Given this orbit, \(\{v_k\}_{k\ge 0}\) is chosen as the gradient value required
for  reproducing \(p^{k+1}\) exactly from
\(p^k\). Finally, \(\{a_k\}_{k\ge 0}\) prescribes the corresponding function values so that
the pairs \(\{(a_k,v_k)\}_{k\ge 0}\) form a compatible first-order Whitney jet.
Throughout the construction, $C>0$ denotes a constant independent of $k$, $K_0$ and $\eta$, whose value may change from line to line.

\medskip\noindent
\textbf{Step 2: the spiral.}
Let $\eta\in(0,1/2]$ be specified later, and choose an integer
$K_0:=K_0(\eta)$ sufficiently large such that
\begin{equation}\label{eq:K0-conditions}
K_0>e^4,\qquad
d_0:=\frac{1}{\sqrt{\log K_0}}\leq\eta^2,
\qquad
\eta\alpha_0d_0^2\leq1.
\end{equation}
For $k\geq0$, define
\begin{eqnarray}\label{dkvark}
&&d_k:=\frac1{\sqrt{\log(k+K_0)}},\qquad r_k:=1-d_k,\nn\\
&&\varphi_0=0,\qquad \beta_k:=\varphi_{k+1}-\varphi_k=\eta\alpha_kd_k^2,\end{eqnarray}
and
\begin{eqnarray}\label{uk}
u(\varphi):=(\cos\varphi,\sin\varphi),\qquad p^k:=r_ku(\varphi_k),\qquad x^k:=(1,p^k).
\end{eqnarray}
We have
\begin{align}
 &d_k\le\eta^2\;\mbox{ for all}\; k\ge0,\label{adddk}\\
&\eta\alpha_kd_k^2\le 1\;\mbox{ for all}\; k\ge 0,\nn\\
 &\sum_{k=0}^\infty\beta_k=\frac\eta2\sum_{k=0}^\infty\frac1{(k+K_0+1)\log(k+K_0)}=\infty,
\;
\mbox{whereas}\; \beta_k\to0.
\end{align}
 Fix \(\vartheta\in[0,2\pi)\).  For every sufficiently large integer \(m\),
let \(k_m\) be the first index such that
\(\varphi_{k_m}\geq2\pi m+\vartheta\).  Then
\[
    0\leq\varphi_{k_m}-(2\pi m+\vartheta)
    \leq\beta_{k_m-1}\longrightarrow0,
\]
so \(p^{k_m}\to(\cos\vartheta,\sin\vartheta)\).  Since \(r_k\to1\), every
cluster point of \(\{p^k\}\) lies on \(\Sone\).  Hence
\[
    \operatorname{cluster}(p^k)=\Sone,
    \qquad
    \operatorname{cluster}(x^k)=\{1\}\times\Sone.
\]
 In addition, we prove the chord estimate stated in Lemma \ref{lem:chord-lower-bound}, i.e.,
\begin{equation}\label{eq:chord}
\norm{p^{k+1}-p^k}\ge\frac{\beta_k}{\pi}.
\end{equation}
Therefore $\sum_{k=0}^{+\infty}\norm{p^{k+1}-p^k}=+\infty$.

\medskip\noindent
\textbf{Step 3: prescribed jet and elementary estimates.}
Define
\begin{eqnarray}\label{vkak}
v_k:=\frac{p^k-p^{k+1}}{\alpha_k},\qquad a_k:=\sum_{\ell=k}^{\infty}\alpha_\ell\norm{v_\ell}^2.
\end{eqnarray}
Then
\[
p^{k+1}=p^k-\alpha_kv_k,\qquad a_{k+1}-a_k=-\alpha_k\norm{v_k}^2.
\]
By \eqref{eq:a-size} of Lemma~\ref{lem:elementary}, we have $a_k<+\infty$.
Set
\[
    C:=\max\{\widehat C_1,\widehat C_2,
             \widehat C_3,\widehat C_4\}.
\]
Then
\begin{eqnarray*}
&&0<d_k-d_{k+1}\leq C\alpha_kd_k^3,\label{eq:d-est}\\
&&\norm{v_k}\leq C\eta d_k^2,\label{eq:v-est}\\
&&\norm{v_{k+1}-v_k}\leq C\eta\norm{p^{k+1}-p^k},\label{eq:vdiff-est}\\
&&0\leq a_k\leq C\eta^2d_k^2.\label{eq:a-est}
\end{eqnarray*}

Let
\[
E:=\Sph^1\cup\{p^k:k\ge0\},
\]
and prescribe
\[
a(q)=0,\quad v(q)=0\quad(q\in\Sph^1),\qquad a(p^k)=a_k,\quad v(p^k)=v_k.
\]

\medskip\noindent
\textbf{Step 4: spiral--circle estimates.}
Fix \(k\geq0\) and \(q\in\Sone\), and set \(R=\norm{p^k-q}\).
Because \(\norm{p^k}=1-d_k\),
\[
   d_k= \dist(p^k,\Sone)\le R.
\]
Using \(d_k\leq1\), \(0<\eta\leq1\), and
Lemma~\ref{lem:elementary}, let
\[
{C_*}:=\max\{\widehat C_1,\widehat C_2,\widehat C_3,\widehat C_4\}.
\]
Then
\begin{align}
    \norm{v(p^k)-v(q)}
    &=
    \norm{v_k}
    \leq C_*\eta d_k^2
    \leq C_*\eta R,
    \label{eq:mixed-gradient}\\
    |a(p^k)-a(q)-\ip{v(q)}{p^k-q}|
    &=
    a_k
    \leq C_*\eta^2d_k^2
    \leq C_*\eta R^2,
    \label{eq:mixed-forward}\\
    |a(q)-a(p^k)-\ip{v(p^k)}{q-p^k}|
    &\leq
    a_k+\norm{v_k}R \nonumber\\
    &\leq
    C_*\eta^2d_k^2+C_*\eta d_k^2R \nonumber\\
    &\leq
    C_*\eta R^2+C_*\eta R^2
    \leq 2 C_*\eta R^2.
    \label{eq:mixed-reverse}
\end{align}
 The last inequalities follow from
\(0<\eta\leq1\), \(d_k\leq R\), and \(d_k\leq1\). Thus, the gradient condition \eqref{eq:whitney-gradient-condition} and both
orientations of the value condition \eqref{eq:whitney-value-condition} are
verified for every spiral--circle pair by taking $C\geq 2C_*.$

\medskip\noindent
\textbf{Step 5: local angular separation spiral--spiral estimates.}
Let $i>j$ and suppose that
\[
0\le \varphi_i-\varphi_j\le \pi.
\]
Define
\begin{eqnarray}\label{Lij}
L_{j,i}:=\sum_{\ell=j}^{i-1}\norm{p^{\ell+1}-p^\ell}.
\end{eqnarray}
Then, Proposition~\ref{prop:local-half-turn} gives
\begin{eqnarray}\label{add1}
\norm{v_i-v_j}\le \widehat C_5\eta\norm{p^i-p^j}.
\end{eqnarray}
Moreover, Proposition~\ref{prop:local-quadratic} yields
\begin{eqnarray}\label{add2}
\left|a_i-a_j-\ip{v_j}{p^i-p^j}\right|
\le
\widehat C_6\eta\norm{p^i-p^j}^2,
\end{eqnarray}
Enlarging $C$ further, take
$C\ge \max(\widehat C_5,\widehat C_6)$.

\medskip\noindent
\textbf{Step 6: large angular separation spiral--spiral estimates.}
If $i>j$ and $\varphi_i-\varphi_j\ge\pi$, then
Lemma~\ref{lem:separation} yields
\[
\norm{p^i-p^j}
\ge
c\max\{d_i,d_j\},
\qquad
c:=1-e^{-\pi/2}>0.
\]
Next, we prove the following estimates.
\begin{eqnarray}\label{eq:different-gradient}
\norm{v_i-v_j}
\le
C\eta\norm{p^i-p^j},
\end{eqnarray}
and
\begin{eqnarray}\label{eq:different-value}
\left|a_i-a_j-\ip{v_j}{p^i-p^j}\right|
\le
C\eta\norm{p^i-p^j}^2.
\end{eqnarray}
First, let
\[
R:=\norm{p^i-p^j},
\qquad
D:=\max\{d_i,d_j\}=d_j.
\]
By Lemma~\ref{lem:separation},
$D\le \frac{R}{c}.$
Since $D\le1$, Lemma~\ref{lem:elementary} gives
\begin{align*}
&\norm{v_i-v_j}
\le
\norm{v_i}+\norm{v_j}
\le
2{\widehat C_2}\eta D^2
\le
2{\widehat C_2}\eta R/c,
\\
&\left|a_i-a_j-\ip{v_j}{p^i-p^j}\right|
\le
a_i+a_j+\norm{v_j}R
\notag\\
&\le
2\widehat C_4\eta^2D^2+\widehat C_2\eta D^2R
\notag\le
\left(\frac{2\widehat C_4}{c^2}+\frac{\widehat C_2}{c}\right)\,\eta R^2,
\end{align*}
where we used \(0<\eta\leq1\). Enlarge $C$ further so that
\[
    C\geq
    \max\left\{
        \frac{2\widehat C_2}{c},
        \frac{2\widehat C_4}{c^2}
        +\frac{\widehat C_2}{c}
    \right\}.
\]
Then \eqref{eq:different-gradient} and \eqref{eq:different-value} hold.

\medskip\noindent
\textbf{Step 7: Whitney extension and exact realization.}
The circle--circle estimates are trivial because \(a=0\) and \(v=0\) on
\(\Sone\).  Equations
\eqref{eq:mixed-gradient}--\eqref{eq:mixed-reverse} treat all mixed pairs.
Equations \eqref{add1}--\eqref{add2} treat
spiral--spiral pairs with angular separation at most \(\pi\), while
\eqref{eq:different-gradient}--\eqref{eq:different-value} treat the remaining
pairs.
The displayed value estimates were written with the earlier spiral point as
the base point.  The reverse orientation follows from
\begin{eqnarray}
    &&a_j-a_i-\ip{v_i}{p^j-p^i}\nonumber\\
    &&=
    -\left(a_i-a_j-\ip{v_j}{p^i-p^j}\right)
    +\ip{v_i-v_j}{p^i-p^j}.
\end{eqnarray}
Together with the corresponding gradient estimates, see
\eqref{eq:mixed-gradient}, \eqref{eq:different-gradient}, and
\eqref{add1},  the preceding estimates show that, after enlarging
\(C>0\) if necessary, the jet \((a,v)\) satisfies
\[
    \|v(p)-v(q)\|\le C\eta\|p-q\|,
\]
and
\[
    |a(p)-a(q)-\langle v(q),p-q\rangle|
    \le C\eta\|p-q\|^2
\]
for all \(p,q\in E\).
Therefore, Corollary~\ref{cor:practical-whitney} yields that there exists
\(\widetilde h_0\in C^{1,1}(\mathbb R^2)\) such that
\[
\widetilde h_0(p^k)=a_k,
\qquad
\nabla\widetilde h_0(p^k)=v_k,
\]
and
\[
\widetilde h_0|_{\mathbb S^1}=0,
\qquad
\nabla\widetilde h_0|_{\mathbb S^1}=0,
\qquad
\operatorname{Lip}(\nabla\widetilde h_0)
\le C_{\mathrm W}C\eta.
\]
The preceding estimates yield a constant \(C>0\), independent of
\(\eta\), \(K_0\), and \(k\), such that the two Whitney bounds above
hold for all \(p,q\in E\). Choose this final constant \(C\) large
enough to dominate all the constants appearing in Steps~3--7.
We then choose
\[
    0<\eta\leq
    \min\left\{\frac12,\frac{1}{C_WC}\right\}.
\] 
Here, $C\ge2\max\{2 C_*,\widehat C_5,\widehat C_6,\frac{2\widehat C_2}{c},
\frac{2\widehat C_4}{c^2}+\frac{\widehat C_2}{c} \}$.
The initial choice of \(K_0=K_0(\eta)\) can now be made sufficiently large so that all the preceding elementary estimates and condition (\ref{eq:K0-conditions}) hold.
Consequently,
\[
\operatorname{Lip}(\nabla\widetilde h_0)\le1.
\]

Add a constant $c_0$ so that $\widetilde h:=\widetilde h_0+c_0\ge1$ on $\overline{\mathbb B}_{\R^2}$, and define
\[
h(x_1,x_2,x_3):=\widetilde h(x_2,x_3).
\]
Initialize $x^0=(1,p^0)$, $z^0=(1,0)$, and $\theta_0>0$. Since $Ax=(1,0)$ for all $x\in\cal S$,
\[
z^{k+1}=\Proj_{[-1,1]^2}\left(\frac1{\gamma_k},0\right)=(1,0).
\]
Together with the initialization $z^0=(1,0)$, this yields  $z^k=(1,0)\neq0$ for all $k\ge0$. Also
\[
A^*z^k=(1,0,0),\qquad \nabla h(x^k)=(0,v_k),\ f\equiv 1,\; y^{k+1}=0.
\]
Therefore, the term  $\frac{\theta_k}{\delta_k}K^* y^{k+1}$ vanishes.
\[
x^k-\frac1{\delta_k}\nabla h(x^k)-\frac1{\delta_k}A^*z^k=(1-\alpha_k,p^{k+1}).
\]
Since $\Proj_{\cal S}(t,p)=(1,\Proj_{\overline{\mathbb B}_{\R^2}}p)$ and $p^{k+1}\in \overline{\mathbb B}_{\R^2}$, the $x$-update gives
\[
x^{k+1}=(1,p^{k+1}).
\]
Thus the prescribed spiral is exactly the primal S-FSPS orbit.

It remains to  verify Assumption \ref{ass:standing}.
\[
A({\cal S})=\{(1,0)\}\subseteq\dom\partial g,\qquad \partial g(1,0)=\{ 1 \}\times[-1,1],
\]
so $\dist(0,\partial g(Ax))=1$ for every $x\in{\cal S}$. Moreover,
\[
g(Ax)+h(x)=1+h(x)\ge2.
\]
Hence Assumption \ref{ass:standing} holds and the cluster-set and infinite-length conclusions now follow from Step 2.

We conclude by verifying that every point in the cluster set is a limiting lifted stationary point.
Let $\overline x=(1,q)$ where $q\in\Sone$. Then,
\begin{eqnarray}\nabla h(\overline x)=0,\;\partial g(A\overline x)=\{1\}\times [-1,1].\end{eqnarray}
$\overline z=(1,0)\in \partial g(A\overline x)$.
Since $(-1,0,0)\in N_{\cal S}(\overline x)$, we have
\begin{eqnarray*}0\in A^*\partial g(A{\overline x})+\nabla h(\overline x)+N_{\cal S}(\overline x).\end{eqnarray*}
Because $f\equiv 1$, this implies that $\overline x$ is  a limiting lifted stationary point.
\end{proof}


\begin{remark}
The counterexample does not contradict the subsequential convergence theorem in
\cite{BotLiTao2025}. It shows that different subsequences may converge to
different stationary cluster points. Moreover, the constructed instance satisfies
\[
    \gamma_k\downarrow0,\qquad
    \sum_{k=0}^{\infty}\gamma_k=+\infty,
    \qquad
    \sum_{k=0}^{\infty}\gamma_k^2<+\infty.
\]
Thus, even square summability of the smoothing parameters does not
ensure whole-sequence convergence.
\end{remark}
\subsection{Full-Dimensional Fractional Construction}
We next construct a full-dimensional counterexample for the fractional model satisfying several natural nondegeneracy properties.

\begin{theorem}\label{thm:full-dimensional-counterexample}
There exists an instance of problem~\eqref{eq:fractional-problem}
satisfying Assumption~\ref{ass:standing} such that
\(\operatorname{int}\mathcal S\neq\varnothing\), \(A\) has full row
rank, and each of \(f\circ K\), \(g\circ A\), and \(g\circ A+h\) is
nonconstant on \(\mathcal S\). Moreover, the S-FSPS iterates satisfy
\[
    \operatorname{cluster}(x^k)=\{1\}\times\mathbb S^1,
    \qquad
    \sum_{k=0}^{\infty}\|x^{k+1}-x^k\|=+\infty.
\]
Furthermore, every point in the cluster set is a limiting lifted stationary point.
Consequently, since the cluster set contains infinitely many points,
the primal sequence does not converge.
\end{theorem}

%

\begin{proof}
Let
\[
\mathcal S:=[1,2]\times \overline{\mathbb B}_{\R^2},
\qquad
g(u):=\|u\|_1 \quad (u\in\mathbb R^2),
\qquad
f(s):=3-s \quad (s\in\mathbb R),
\]
and define
\[
A=\frac13
\begin{pmatrix}
2&1&0\\
2&0&1
\end{pmatrix},
\qquad
K=
\begin{pmatrix}
1&0&0
\end{pmatrix},
\qquad
b=
\begin{pmatrix}
4/3\\
1/3\\
1/3
\end{pmatrix}.
\]
Then
$\sigma_A=1$ and $\operatorname{rank}(A)=2.$
Moreover, for every \(x\in\mathcal S\),
$
Ax\geq
\begin{pmatrix}
1/3\\
1/3
\end{pmatrix},$
$$g(Ax)=\|Ax\|_1=\frac{4}{3}x_1+\frac{1}{3}x_2+\frac{1}{3}x_3=\langle b,x\rangle.$$

\noindent Choose \(\chi=2\) and set
$
\gamma_k:=\frac{1}{k+K_0}.
$
We shall construct a function \(h\) satisfying
\[
\operatorname{Lip}(\nabla h)\leq 1.
\]
Accordingly, we set
$
\delta_k:=2(k+K_0+1)$ where $K_0\ge e^4>4$
and
$\alpha_k:=\frac{1}{\delta_k}.$
Choose $\eta>0$ and $K_0=K_0(\eta)$ exactly as in Step~7 of the
proof of Theorem~\ref{thm:main-counterexample}; in particular,
\[
    0<\eta\leq
    \min\left\{\frac12,\frac{1}{C_WC}\right\},
\]
and $K_0$ is sufficiently large so that \eqref{eq:K0-conditions}
holds.
Then, define \(d_k\), \(r_k\), \(\beta_k\), \(p^k\),
and \(x^k\) as in \eqref{dkvark} and \eqref{uk}, respectively.
Moreover, let \(a_k\) and \(v_k\) be defined as in \eqref{vkak}.

Following Steps 2--6 and the Whitney-extension part of Step 7 in the
proof of Theorem 4.1, and applying Corollary 3.2, we obtain
\(h_0\in C^{1,1}(\mathbb R^2)\) such that
\[
h_0(p^k)=a_k,
\qquad
\nabla h_0(p^k)=v_k,
\]
and
$h_0|_{\Sph^1}=0$,
$
\nabla h_0|_{\Sph^1}=0$,
$\operatorname{Lip}(\nabla h_0)\leq 1.
$
Choose $c_0>0$ sufficiently large so that
\[
    \widehat h(u):=h_0(u)+c_0\geq 3
    \qquad\text{for all }u\in\mathbb B_{\mathbb R^2}.
\]
Define
\[
    h(x_1,x_2,x_3)
    :=\widehat h(x_2,x_3)-\langle b,x\rangle.
\]
Thus, $\operatorname{Lip}(\nabla h)\leq 1$.
Initialize $x^0=(1,p^0)$, $z^0=(1,1)$, and $\theta_0>0$. Since $\gamma_k<1/3$ and $Ax^{k+1}\ge(1/3,1/3)$, it follows that
$z^{k+1}=\Proj_{[-1,1]^2}\left(\frac{A x^{k+1}}{\gamma_k}\right)=(1,1).$
Thus, $z^k=(1,1)\neq0$ for all $k\ge 0$. Also
\[
A^*z^k=b,\qquad \nabla h(x^k)=(0,v_k)-b,\ y^{k+1}=-1.
\]
Therefore,
\[
x^k+\frac{\theta_k}{\delta_k}K^* y^{k+1}-\frac1{\delta_k}\nabla h(x^k)-\frac1{\delta_k}A^*z^k=(1-\alpha_k\theta_k,p^{k+1}).
\]

We prove by induction that, for every $k\ge0$, the following two statements hold:
\begin{itemize}
\item[$\large{\textcircled{\small{1}}}_{k}$] $x^k=(1,p^k)$;
\item[$\large{\textcircled{\small{2}}}_{k}$] $\theta_k>0$.
\end{itemize}
The base case follows from the initialization $x^0=(1,p^0)$ and $\theta_0>0$.
Suppose that $\large{\textcircled{\small{1}}}_{k}$ and $\large{\textcircled{\small{2}}}_{k}$ hold.
Note that $x^{k+1}=(\Proj_{[1,2]}(1-\alpha_k\theta_k),\Proj_{\overline{\mathbb B}_{\R^2}}(p^{k+1}))=(1,p^{k+1})$
because $\theta_k>0$ and $p^{k+1}\in\overline{\mathbb B}_{\R^2}$.
Thus, $\large{\textcircled{\small{1}}}_{k+1}$ holds.

Moreover, since
\(g(Ax)=\langle b,x\rangle\) and
\(h(x)=\widehat h(x_2,x_3)-\langle b,x\rangle\), we have
$
\Psi(x^{k+1},z^{k+1};\gamma_k)
=\widehat h(p^{k+1})-\gamma_k>0,
\;
f(Kx^{k+1})=2.$ Hence,
$
\theta_{k+1}
=\frac{\widehat h(p^{k+1})-\gamma_k}{2}>0,
$
so \(\textcircled{\small 2}_{k+1}\) holds.
Thus the prescribed spiral is exactly the primal S-FSPS orbit.

It remains to  verify Assumption \ref{ass:standing}.
   $ K(\mathcal S)=[1,2]
    \subseteq \operatorname{int}(\operatorname{dom}f),$
    $f(Kx)=3-x_1\in[1,2],$
$
    \operatorname{rank}(A)=2,
    \;
    \|A\|=1,$
$A(\mathcal S)\subseteq
\{u\in\mathbb R^2:u_1>0,\ u_2>0\}.$
   For any $x\in{\cal S}$,
    $\partial g(Ax)=\{(1,1)^\top\}$.
Hence, $\dist(0,\partial g(Ax))=\sqrt{2}$ for every $x\in{\cal S}$. Moreover,
\[
g(Ax)+h(x)=\widehat h(x_2,x_3)\ge 3.
\]
Hence, all the standing assumptions are satisfied.

We conclude by verifying that every point in the cluster set is a limiting lifted stationary point.
Let \(\overline x=(1,q)\), where \(q\in\mathbb S^1\), and set
\(e_1:=(1,0,0)^\top\).

Since
$
\partial g(A\overline x)=\{(1,1)^\top\},\;
A^*(1,1)^\top=b,\;
\nabla h(\overline x)=-b,
$
we have
$
A^*(1,1)^\top+\nabla h(\overline x)=0.
$
Moreover,
$
f(K\overline x)=2,\;
g(A\overline x)+h(\overline x)=c_0,\;
K^*\partial f(K\overline x)=\{-e_1\}.
$
Because \(\overline x_1=1\), we have
$
-\frac{c_0}{2}e_1\in N_{\mathcal S}(\overline x).
$
Consequently,
$$
2\left(
A^*(1,1)^\top+\nabla h(\overline x)-\frac{c_0}{2}e_1
\right)
-c_0(-e_1)=0.
$$
Therefore,
$
0\in
f(K\overline x)\bigl(
A^*\partial g(A\overline x)+\nabla h(\overline x)+N_{\mathcal S}(\overline x)
\bigr)
-\bigl(g(A\overline x)+h(\overline x)\bigr)K^*\partial f(K\overline x),
$
so \(\overline x\) is a limiting lifted stationary point.
Moreover, the numerator is nonconstant on $\mathcal S$. Indeed,
$g(Ax)+h(x)=\widehat h(x_2,x_3).$
For every $q\in\mathbb S^1$, we have $\widehat h(q)=c_0$, whereas
$
\widehat h(p^k)=c_0+a_k>c_0,
$
where we have used $a_k>0$ for $k\geq 0$. Indeed,
$
r_{k+1}\neq r_k
\;\Longrightarrow\;
p^{k+1}\neq p^k
\;\Longrightarrow\;
v_k\neq 0,
$
and hence $a_k>0$ by \eqref{vkak}. Therefore, $g\circ A +h $ is nonconstant on \(\mathcal S\).
Since \(f(Kx)=3-x_1\) and \(g(Ax)=\langle b,x\rangle\),
both \(f\circ K\) and \(g\circ A\) are nonconstant on \(\mathcal S\).
\end{proof}
\section{Nonconvex  Composite Optimization Specialization}\label{sec:com}


The same phenomenon  occurs for the  composite
problem.
\begin{equation}\label{eq:PP}
    \min_{x\in\mathcal S}\; g(Ax)+h(x),
\end{equation}
where $\mathcal S\subseteq\mathbb R^n$ is nonempty, convex and compact;
$g:\mathbb R^s\to\overline{\mathbb R}$ is proper, convex and lower
semicontinuous; $A:\mathbb R^n\to\mathbb R^s$ is linear; and
$h:\mathbb R^n\to\mathbb R$ is differentiable on an open set containing
$\mathcal S$, with a Lipschitz continuous gradient. Choose $\chi>1$, let $\{\gamma_k\}_{k\ge0}$ be positive and
nonincreasing and satisfy~\eqref{eq:gammak}, and define
$\delta_k$ by~\eqref{eq:delta-general}.
Starting from $x^0\in\mathcal S$ and $z^0\in\mathbb R^s$,
generate, for $k\geq 0$,
\begin{align}
x^{k+1}
&=
\Proj_{\mathcal S}\left(
x^k-\frac{1}{\delta_k}
\bigl(\nabla h(x^k)+A^*z^k\bigr)
\right),
\label{eq:reduced-x}\\
z^{k+1}
&=
\prox_{g^*/\gamma_k}
\left(\frac{Ax^{k+1}}{\gamma_k}\right).
\label{eq:reduced-z}
\end{align}

The following result establishes a subsequential convergence guarantee for
this reduced scheme.

\begin{theorem}\label{cor:reduced-subsequential}
Assume that $\mathcal S$ is nonempty,  convex, and compact; that
$h$ is differentiable on an open neighborhood of $\mathcal S$ and
$\nabla h$ is $L_{\nabla h}$-Lipschitz continuous there for some
$L_{\nabla h}>0$; that $g:\mathbb R^s\to\overline{\mathbb R}$ is proper, convex, and lower
semicontinuous; that
$A(\mathcal S)\subseteq\dom(\partial g)$; and that there exists
$\ell>0$ such that
$
    \dist\bigl(0,\partial g(Ax)\bigr)\leq\ell
    \;\text{for every }x\in\mathcal S.
$
Then every sequence $\{(x^k,z^k)\}_{k\ge 0}$ generated by
\eqref{eq:reduced-x}--\eqref{eq:reduced-z} satisfies
\begin{equation}\label{eq:reduced-liminf}
    \liminf_{k\to\infty}
    \delta_k\|x^{k+1}-x^k\|=0.
\end{equation}
Moreover, there exist a strictly increasing sequence of indices
\(\{k_j\}_{j\geq1}\) and a pair
\((\bar x,\bar z)\in\mathcal S\times\mathbb R^s\) such that
$
    x^{k_j}\to\bar x,\; z^{k_j}\to\bar z,
$
 and
\begin{equation}\label{eq:reduced-lifted-stationarity}
    \overline z\in\partial g(A\overline x),
    \qquad
    0\in
    \nabla h(\overline x)+A^*\overline z+N_{\mathcal S}(\overline x).
\end{equation}
Consequently, $\overline x$ is an exact stationary point of \eqref{eq:PP}, in
the sense that
\begin{equation}\label{eq:reduced-stationarity}
    0\in
    \nabla h(\overline x)
    +A^*\partial g(A\overline x)
    +N_{\mathcal S}(\overline x).
\end{equation}

\end{theorem}

\begin{proof}
Invoking Lemma~\ref{lem:uniform-moreau-estimates}(i) together with the
dual update~\eqref{eq:reduced-z}, we obtain
\begin{eqnarray}\label{zkbound}
    \|z^{k+1}\|
    =\|\nabla g_{\gamma_k}(A x^{k+1})\|\le \ell.
\end{eqnarray}
Hence, the dual sequence \(\{z^k\}\) is bounded.
Moreover,
$
    A(\mathcal S)\subseteq\dom(\partial g)\subseteq\dom g.
$
Thus $g\circ A+h$ is finite-valued and lower semicontinuous on the compact
set $\mathcal S$, and therefore
$
    \inf_{x\in\mathcal S}\bigl(g(Ax)+h(x)\bigr)>-\infty.
$

Define $\Psi_k:=\Psi(x^k,z^k;\gamma_{k-1})$ for $k\ge 1$. By the dual update (\ref{eq:reduced-z}) at \((k-1)\)th iteration,
\begin{eqnarray}\label{zkop}
    z^k=\nabla g_{\gamma_{k-1}}(Ax^k),
\end{eqnarray}
and hence
\[
    \Psi_k=g_{\gamma_{k-1}}(Ax^k)+h(x^k).
\]

Let $d^k:=x^{k+1}-x^k$ and define
$
    F_k:=g_{\gamma_{k-1}}\circ A+h,
    \;
    L_k:=L_{\nabla h}+\frac{\sigma_A^2}{\gamma_{k-1}}.
$
Since
the update \eqref{eq:reduced-x} is a
projected-gradient step for $F_k$ with stepsize $1/\delta_k$,
 the descent lemma yields
\[
    F_k(x^{k+1})
    \leq
    F_k(x^k)
    -
    \frac{\delta_k-L_k}{2}\|d^k\|^2.
\]
Since $\gamma_k\leq\gamma_{k-1}$, we have
\[
\begin{aligned}
    \delta_k-L_k
    =
    \chi\left(
        L_{\nabla h}+\frac{\sigma_A^2}{\gamma_k}
    \right)
    -
    \left(
        L_{\nabla h}+\frac{\sigma_A^2}{\gamma_{k-1}}
    \right)                                                  
    \geq
    (\chi-1)
    \left(
        L_{\nabla h}+\frac{\sigma_A^2}{\gamma_k}
    \right).
\end{aligned}
\]
Consequently,
\[
    g_{\gamma_{k-1}}(Ax^{k+1})+h(x^{k+1})
    \leq
    g_{\gamma_{k-1}}(Ax^k)+h(x^k)
    -
    \widetilde c_k\|d^k\|^2,
\]
where
\[
    \widetilde c_k
    :=
    \frac{\chi-1}{2}
    \left(
        L_{\nabla h}+\frac{\sigma_A^2}{\gamma_k}
    \right)
    =
    \frac{\chi-1}{2\chi}\delta_k
    >0.
\]
Moreover, the dual representation of the Moreau envelope and the
definition of $z^{k+1}$ imply
\[
    g_{\gamma_k}(Ax^{k+1})
    \leq
    g_{\gamma_{k-1}}(Ax^{k+1})
    +
    \frac{\gamma_{k-1}-\gamma_k}{2}
    \|z^{k+1}\|^2.
\]
Combining the preceding two inequalities gives
\begin{eqnarray}\label{despsi}
    \Psi_{k+1}
    \leq
    \Psi_k-\widetilde c_k\|x^{k+1}-x^k\|^2
    +\Xi^{k+1},
\end{eqnarray}
 where
\begin{eqnarray*}
\Xi^{k+1} := \frac{\gamma_{k-1}-\gamma_{k}}{2}\|{z}^{k+1}\|^2 \geq 0.
\end{eqnarray*}
\noindent Lemma~\ref{lem:uniform-moreau-estimates}(iii) gives the lower bound
\[
    g_{\gamma_{k-1}}\!\left(Ax^k\right)
    \geq
    g\!\left(Ax^k\right)
    -
    \frac{\gamma_{k-1}\ell^2}{2},
\]
so $\{\Psi_k\}_{k\ge 1}$ is bounded below. Combining (\ref{zkbound}), (\ref{despsi}), and the fact that
$\gamma_{k-1}\geq\gamma_k$ yields
$$
    \sum_{k=1}^{\infty} \Xi^{k+1}
    \leq
    \frac{\ell^2}{2}
    \sum_{k=1}^{\infty}\left(\gamma_{k-1}-\gamma_k\right)
    <+\infty.
$$

Consequently,
\begin{eqnarray}\label{addinf}\sum_{k=1}^{\infty}
    \delta_k\left\|x^{k+1}-x^k\right\|^2
    <+\infty.\end{eqnarray}
Since \(L_{\nabla h}>0\),
$\frac{1}{\delta_k} =
    \frac{\gamma_k}
    {\chi\left(L_{\nabla h}\gamma_k+\sigma_A^2\right)}
    \geq
    \frac{\gamma_k}
    {\chi\left(L_{\nabla h}\gamma_0+\sigma_A^2\right)}.$
    Thus,
$\sum_{k=1}^{\infty}\gamma_{k}=+\infty
    \;\Longrightarrow\;
    \sum_{k=1}^{\infty}\frac{1}{\delta_k}=+\infty.$
It then follows that
$$
    \liminf_{k\to+\infty}
    \delta_k\left\|x^{k+1}-x^k\right\|
    =0.
$$
Suppose not. Then there exist $\varepsilon>0$ and $\widehat K$ such that
$\delta_k\|x^{k+1}-x^k\|\ge\varepsilon$ for all $k\ge\widehat K$. Hence
$
  \delta_k^2\,\|x^{k+1}-x^k\|^2 \;\ge\; \varepsilon^2$
  for all $k\ge\widehat K,$
and summing over $k\ge\widehat K$ gives
\begin{eqnarray*}
  \sum_{k\ge\widehat K}\frac{[\delta_k^2\,\|x^{k+1}-x^k\|^2]}{\delta_k}
  \;\ge\; \varepsilon^2\sum_{k\ge\widehat K}\frac{1}{\delta_k} \;=\; +\infty,
\end{eqnarray*}
which contradicts (\ref{addinf}).

\noindent Choose a subsequence $\{k_j\}_{j\in\mathbb N}$ such that
$ \delta_{k_j}\|x^{k_j+1}-x^{k_j}\|\to0.$
By the compactness of $\mathcal S$ and the boundedness of the dual sequence,
we may pass to a further subsequence, without relabeling, such that
$x^{k_j}\to\overline x,\; z^{k_j}\to\overline z.$
Since $\delta_k$ is bounded away from zero under
\eqref{eq:delta-general}, we also have
$x^{k_j+1}-x^{k_j}\to0
    \;\text{and hence}\;
    x^{k_j+1}\to\overline x.$
The optimality condition for the primal update is
\begin{equation}\label{eq:reduced-x-optimality}
   \xi_j\in N_{\mathcal S}(x^{k_j+1}),
\end{equation}
where $\xi_j:=-[ \nabla h(x^{k_j})
    +A^*z^{k_j}
    +\delta_{k_j}(x^{k_j+1}-x^{k_j})]$.
On the other hand, the dual update at iteration $k_j-1$ gives
\begin{equation}\label{eq:reduced-z-optimality}
    Ax^{k_j}-\gamma_{k_j-1}z^{k_j}
    \in\partial g^*(z^{k_j}).
\end{equation}
Passing to the limit in \eqref{eq:reduced-x-optimality} and
\eqref{eq:reduced-z-optimality}, using $\gamma_{k_j-1}\to0$, $\xi_j\to -[\nabla h(\overline x)+A^*\overline z]$ and the closedness of
the graphs of $N_{\mathcal S}$ and $\partial g^*$ \cite[Proposition~8.7]{RockWets}, yields
$
    0\in
    \nabla h(\overline x)+A^*\overline z+N_{\mathcal S}(\overline x)
$
and
$
    A\overline x\in\partial g^*(\overline z).
$
The latter relation is equivalent to
$\overline z\in\partial g(A\overline x)$. This proves
\eqref{eq:reduced-lifted-stationarity}, and
\eqref{eq:reduced-stationarity} follows immediately.
\end{proof}
\begin{corollary}\label{cor:counterexample-consequences}

Apply \eqref{eq:reduced-x}--\eqref{eq:reduced-z}  to the instance constructed in Theorem
\ref{thm:main-counterexample}, with the same choices
\[
    \chi=2,\qquad
    \gamma_k=\frac{1}{k+K_0},\qquad
    \delta_k=2(k+K_0+1),
\]
and initialize
\[
    x^0=(1,p^0),\qquad z^0=(1,0).
\]
Then the following properties hold:
\begin{enumerate}[label=\textup{(\roman*)}]
\item The cluster set of the primal sequence is given by
\[
    \operatorname{cluster}(x^k)
    =
    \{1\}\times\mathbb S^1 .
\]
Moreover, every cluster point of the primal sequence is an exact stationary point of~\eqref{eq:PP} in the sense of~\eqref{eq:reduced-stationarity}.

\item The generated trajectory has infinite length, namely,
\[
    \sum_{k=0}^{\infty}\|x^{k+1}-x^k\|
    =
    +\infty .
\]
The primal sequence does not converge.

\end{enumerate}
\end{corollary}
\begin{proof}
Since \(f\equiv1\), the S-FSPS primal and dual updates reduce to
\eqref{eq:reduced-x}--\eqref{eq:reduced-z}.
The conclusions  follow directly from Theorem~\ref{thm:main-counterexample}.
\end{proof}

\section{Conclusion}\label{sec:conclusion}

This paper shows that the standing assumptions of S-FSPS do not
guarantee convergence of the entire primal sequence. Using a
compatible first-order jet and a $C^{1,1}$ Whitney extension theorem,
we construct two admissible instances whose primal iterates follow
slowly rotating spirals, each with cluster set
$\{1\}\times\mathbb S^1$. The resulting sequences have infinite
length and fail to converge, although every cluster point is an exact
limiting lifted stationary point.  This phenomenon persists when the feasible set is
full-dimensional, the linear operator  $A$ has full row rank,
and both the numerator and the denominator are nonconstant. The construction
also yields a nonconvergent example for the corresponding variable-smoothing,
single-loop, full-splitting method for  nonconvex composite optimization,
despite the existence of a subsequence converging to an exact stationary
point. Thus, the existing convergence theory is sharp with respect to
whole-sequence convergence: subsequences may converge to exact stationary
points while the full sequence fails to converge. The counterexamples do not arise from unboundedness or from
nonstationarity of their cluster points; rather, they show that
stationarity of all cluster points is insufficient to control the
global geometry of the generated trajectories.




\section*{Statements and Declarations}

\paragraph{Funding.} The research of Min Tao was supported  by the National Natural Science Foundation of China (Grant No. 12471289).

\paragraph{Competing interests.}
The author has no relevant financial or non-financial interests to disclose.

\paragraph{Data availability.}
No data were generated or analyzed in this study.

\appendix
\section{Proof of Lemma \ref{lem:uniform-moreau-estimates}}\label{appA}

\begin{proof}
Fix $w\in A(\mathcal S)$. Since $\partial g(w)$ is a nonempty closed convex
set, there exists $v_w\in\partial g(w)$ satisfying
$
\|v_w\|=\dist(0,\partial g(w))\leq\ell.$
Set $p:=\prox_{\gamma g}(w)$ and
$q:=\nabla g_\gamma(w)=(w-p)/\gamma$. Because
$q\in\partial g(p)$ and $v_w\in\partial g(w)$, monotonicity of $\partial g$
gives
$
\langle q-v_w,p-w\rangle\geq0.$
Since $p-w=-\gamma q$, it follows that
$
\|q\|^2\leq\langle v_w,q\rangle
\leq\|v_w\|\,\|q\|,$
which proves (i). (ii) follows immediately from
$p-w=-\gamma q$.
The subgradient inequality gives
$
g(u)\geq g(w)+\langle v_w,u-w\rangle.
$
Consequently,
\begin{eqnarray*}
g_\gamma(w)
\geq g(w)+\inf_{d\in\R^s}
\left\{\langle v_w,d\rangle+\frac{1}{2\gamma}\|d\|^2\right\}
=g(w)-\frac{\gamma}{2}\|v_w\|^2
\geq g(w)-\frac{\gamma\ell^2}{2}.
\end{eqnarray*}
Together with $g_\gamma(w)\leq g(w)$, this proves
(iii).
Finally, let $u,v\in A(\mathcal S)$, and choose
$v_u\in\partial g(u)$ and $v_v\in\partial g(v)$ with
$\|v_u\|,\|v_v\|\leq\ell$. Applying the subgradient inequality at $u$ and
at $v$ gives
$
g(v)\geq g(u)+\langle v_u,v-u\rangle,$ and
$
g(u)\geq g(v)+\langle v_v,u-v\rangle.
$
These two inequalities imply (iv).
\end{proof}

\section{Proof of Theorem \ref{PriTheo2R}}\label{appB}
\begin{proof}In \cite[Theorem~4.3]{BotLiTao2025}, the additional condition
$
    A(\mathcal S)\subseteq\operatorname{int}(\operatorname{dom}g)
$
is imposed only for assertions \textup{(vi)}--\textup{(vii)};
the proofs of assertions \textup{(i)}--\textup{(v)} do not use
this condition. It therefore remains  to reprove
\textup{(vi)} and \textup{(vii)}.\\
\textup{(vi)}
Let $(\overline x,\overline y,\overline z)\in\Omega$ and choose a subsequence such that
$
(x^{k_j},y^{k_j},z^{k_j})\longrightarrow(\overline x,\overline y,\overline z).
$
Set $\widehat\gamma_j:=\gamma_{k_j-1}$. Then $\widehat\gamma_j\to0$ and
$Ax^{k_j},A\overline x\in A(\mathcal S)$. By
Lemma~\ref{lem:uniform-moreau-estimates} (iii) and (iv),
\begin{align}
|g_{\widehat\gamma_j}(Ax^{k_j})-g(A\overline x)|
&\leq
|g_{\widehat\gamma_j}(Ax^{k_j})-g(Ax^{k_j})|
+|g(Ax^{k_j})-g(A\overline x)|\nonumber\\
&\leq\frac{\widehat\gamma_j\ell^2}{2}
+\ell\|Ax^{k_j}-A\overline x\|
\longrightarrow0.
\label{eq:moreau-value-convergence}
\end{align}
Since
$
\Psi(x^{k_j},z^{k_j};\widehat \gamma_j)
=g_{\widehat \gamma_j}(Ax^{k_j})+h(x^{k_j})
$
and
$
\theta_{k_j}
=\frac{\Psi(x^{k_j},z^{k_j};\widehat \gamma_j)}
{f(Kx^{k_j})},
$
letting $j\to\infty$ and using the continuity of $h$ on $\mathcal S$ and of $f$ on
$K(\mathcal S)$ yields
$
\frac{g(A\overline x)+h(\overline x)+\iota_{\mathcal S}(\overline x)}
{f(K\overline x)}=\overline\theta.
$
This proves \textup{(vi)}.

\textup{(vii)}
Let \(\{{ x}^{k_j}\}\) be a subsequence of $\{ x^k\}$ such that
\begin{eqnarray}\label{add11}
\lim\limits_{j \to +\infty} \delta_{k_j} \|{ x}^{k_j+1} - { x}^{k_j}\| = 0
 \end{eqnarray}
 and let $\overline{ x} \in {\cal S}$ be a cluster point of it. Then there exists a further subsequence $\{{ x}^{k_s}\}$ of $\{{ x}^{k_j}\}$ that converges to ${\overline{ x}}$ as $s \to +\infty$. Since $\delta_k\ge\chi L_{\nabla h}$ for all $k \geq 0$, we have $\lim_{s\to+\infty}\|{ x}^{k_s+1}-{ x}^{k_s}\|=0$, thus ${ x}^{k_s+1} \rightarrow \overline { x}$ as $s \to +\infty$.
By assertion (ii), $\{y^k\}$ is bounded. Passing to a further
subsequence if necessary, suppose that
$y^{k_s+1}\to\overline y$. Since
$y^{k_s+1}\in\partial f(Kx^{k_s})$ and
$Kx^{k_s}\to K\overline x$, the closedness of
$\operatorname{gph}\partial f$ yields
$\overline y\in\partial f(K\overline x)$.
From  the $ x$-update in (\ref{eq:x-update}) and $\nabla(g_{\gamma_{k_s-1}} \circ A)({x}^{k_s})=A^* \nabla g_{\gamma_{k_s-1}}(A {x}^{k_s})$, for all $s \geq 0$, there exists ${ \xi}^{k_s+1} \in \partial \iota_{\cal S}({ x}^{k_s+1})$ such that
\begin{eqnarray}\label{KKT9}
\ \ \ \ \ \ \ \ 0 = { \xi}^{k_s+1}  +\delta_{k_s}({ x}^{k_s+1}-{ x}^{k_s}) + A^*\nabla g_{\gamma_{k_s-1}}(A{ x}^{k_s}) + \nabla h({ x}^{k_s})  - \theta_{k_s} K^* { y}^{k_s+1}.
\end{eqnarray}

Recall
$\widehat \gamma_s:=\gamma_{k_s-1},$  and let $q_s:=\nabla g_{\widehat\gamma_s}(Ax^{k_s}),$ and $p_s:=\prox_{\widehat\gamma_s g}(Ax^{k_s}).$
By  Lemma \ref{lem:uniform-moreau-estimates}(i), $\|q_s\|\leq\ell$. Passing to a
subsequence, suppose that $q_s\to\overline q$. Moreover,
$
\|p_s-Ax^{k_s}\|
\leq\widehat\gamma_s\ell\longrightarrow0,
$
so $p_s\to A {\overline x}$. Since \(q_s\in\partial g(p_s)\), \(p_s\to A{\overline x}\), and \(q_s\to \overline{q}\),
the closedness of \(\operatorname{gph}\partial g\) yields
\(\overline q\in\partial g(A\overline{x})\).
All terms in \eqref{KKT9}, except possibly
$\xi^{k_s+1}$, are bounded; hence $\{\xi^{k_s+1}\}$ is bounded. Passing
to a further subsequence $\xi^{k_s+1}\to \overline \xi$, and using the closedness of
$\operatorname{gph}(\partial\iota_{\mathcal S})$, we obtain
$\overline\xi\in\partial\iota_{\mathcal S}(\overline x)$.
Moreover, combining with $\theta_{k_s}\to\overline \theta$ due to (iv), (\ref{add11}) and (\ref{KKT9}), we have
$
0=\overline\xi+A^*\overline q+\nabla h(\overline x)
-\overline\theta K^*\overline y.
$
Therefore,
$
0\in\partial\iota_{\mathcal S}(\overline x)
+A^*\partial g(A\overline x)+\nabla h(\overline x)
-\overline\theta K^*\partial f(K\overline x).
$
Moreover, assertion~\textup{(vi)} gives
\[
    \overline\theta=\frac{g(A\overline x)+h(\overline x)}{f(K\overline x)}.
\]
Multiplying the preceding inclusion by \(f(K\overline x)>0\) therefore shows
that \(\overline x\) is a limiting lifted stationary point of
\eqref{eq:fractional-problem}.
\end{proof}
\section{Proof of the chord estimate}\label{app:chord}
\begin{lemma}
\label{lem:chord-lower-bound}
For every \(k\ge 0\),
\[
\|p^{k+1}-p^k\|
\ge \frac{\beta_k}{\pi}.
\]
\end{lemma}
\begin{proof}
Since the angle between $u(\varphi_{k+1})$ and $u(\varphi_k)$ is $\beta_k$,
\[
\langle u(\varphi_{k+1}),u(\varphi_k)\rangle=\cos\beta_k.
\]
Consequently,
\begin{align*}
\|p^{k+1}-p^k\|^2
&=
\|r_{k+1}u(\varphi_{k+1})-r_ku(\varphi_k)\|^2\\
&=
r_{k+1}^2+r_k^2-2r_kr_{k+1}\cos\beta_k\\
&=
(r_{k+1}-r_k)^2
+4r_kr_{k+1}\sin^2\!\left(\frac{\beta_k}{2}\right).
\end{align*}
Dropping the first nonnegative term gives
\[
\|p^{k+1}-p^k\|
\ge
2\sqrt{r_kr_{k+1}}
\sin\!\left(\frac{\beta_k}{2}\right).
\]
Since $r_k,r_{k+1}\ge 1/2$,
$
2\sqrt{r_kr_{k+1}}\ge1,
$
and therefore
$
\|p^{k+1}-p^k\|
\ge
\sin\!\left(\frac{\beta_k}{2}\right).
$
For $0\le t\le\pi/2$, concavity of $\sin$ gives
$
\sin t\ge \frac{2}{\pi}t.
$
Because $0<\beta_k\le1<\pi$ for all $k\ge0$, we have
$0<\beta_k/2<\pi/2$, and hence
\[
\sin\!\left(\frac{\beta_k}{2}\right)
\ge
\frac{\beta_k}{\pi}.
\]
Thus
$
\|p^{k+1}-p^k\|\ge\frac{\beta_k}{\pi}.
$
\end{proof}

\section{Proof of the elementary spiral estimates}\label{app:spiral-estimates}

\begin{lemma}
\label{lem:elementary}
Fix \(\eta\in(0,1/2]\) and choose \(K_0\) such that
\eqref{eq:K0-conditions} holds. Then there exist absolute constants
\(\widehat C_1,\widehat C_2,\widehat C_3,\widehat C_4>0\), independent
of \(k\), \(K_0\), and \(\eta\), such that
\begin{align}
    0<d_k-d_{k+1}
    &\leq \widehat C_1\alpha_kd_k^3,
    \label{eq:d-first}\\
    \norm{v_k}
    &\leq \widehat C_2\eta d_k^2,
    \label{eq:v-size}\\
    \norm{v_{k+1}-v_k}
    &\leq \widehat C_3\eta\,\norm{p^{k+1}-p^k},
    \label{eq:v-consecutive}\\
    0\leq a_k
    &\leq \widehat C_4\eta^2d_k^2.
    \label{eq:a-size}
\end{align}
\end{lemma}
\begin{proof}
Set
$
    u_k:=u(\varphi_k),
    \;
    e_k:=(-\sin\varphi_k,\cos\varphi_k),
    \;
    \beta_k:=\varphi_{k+1}-\varphi_k
             =\eta\alpha_kd_k^2.
$
Choose \(K_0\) such that
\eqref{eq:K0-conditions} holds. Then, for every \(k\geq0\),
\[
r_k\geq\frac12,
\qquad
0\leq\beta_k\leq1,
\qquad
d_k^2\leq\eta,
\qquad
d_k^3\leq\eta^2.
\label{eq:app-K-conditions}
\]

\noindent To prove \eqref{eq:d-first}, define
\[
d(t):=(\log t)^{-1/2},\qquad t\geq K_0.
\]
Since
\[
d'(t)
=-\frac{1}{2t(\log t)^{3/2}}
=-\frac{d(t)^3}{2t},
\]
and \(d_k=d(k+K_0)\), the mean-value theorem implies that, for some
\(\xi_k\in(k,k+1)\),
\begin{align*}
0<d_k-d_{k+1}
&=-d'(\xi_k+K_0)\\
&=\frac{1}{2(\xi_k+K_0)
[\log(\xi_k+K_0)]^{3/2}}\\
&\leq
\frac{1}{2(k+K_0)[\log(k+K_0)]^{3/2}}\\
&\leq
\frac{2}{2(k+K_0+1)[\log(k+K_0)]^{3/2}}\\
&=2\alpha_k d_k^3,
\end{align*}
where the second inequality follows from \(k+K_0+1\leq 2(k+K_0)\).
Thus, \eqref{eq:d-first} holds with
\(\widehat C_1=2\).

Note that
\begin{equation}
    \left|d_{k+1}^2-d_k^2\right|=\left|d_{k+1}-d_k\right|\left|d_{k+1}+d_k\right|\le 2d_k\left|d_{k+1}-d_k\right|
    \leq 4\alpha_kd_k^4.
    \label{eq:app-d-square}
\end{equation}

Define the scaled radial increment
\[
    \rho_k
    :=\frac{r_{k+1}-r_k}{\alpha_k}
    =\frac{d_k-d_{k+1}}{\alpha_k}.
\]
We have that
\begin{equation}
    0\leq \rho_k\leq 2d_k^3.
    \label{eq:app-rho-size}
\end{equation}
Define
$
    F(t):=2(t+1)\bigl(d(t)-d(t+1)\bigr).$
Then \(\rho_k=F(k+K_0)\).  Since
\[
    d(t)-d(t+1)
    =\frac12\int_t^{t+1}
      \frac{ds}{s(\log s)^{3/2}},
\]
we can write
\[
    F(t)=(t+1)\int_t^{t+1}f(s)\,ds,
    \qquad
    f(s):=\frac{1}{s(\log s)^{3/2}}.
\]
Differentiation gives
\[
    F'(t)
    =\int_t^{t+1}\bigl(f(s)+(t+1)f'(s)\bigr)\,ds,
\]
where
\[
    f'(s)
    =-\frac{1}{s^2(\log s)^{3/2}}
     -\frac{3}{2s^2(\log s)^{5/2}}.
\]
Using the preceding representation of $F'(t)$, together with the mean-value theorem, there exists $\widehat s\in(t,t+1)$ such that
\begin{eqnarray*}
|F'(t)|
&\leq& \frac{1}{\widehat s(\log(\widehat s))^{3/2}}
+(t+1)\left(\frac{1}{\widehat s^2(\log(\widehat s))^{3/2}}
+\frac{3}{2\widehat s^2(\log(\widehat s))^{5/2}}\right)\nn\\
&\leq&\frac{1}{t(\log(t))^{3/2}}
+(t+1)\left(\frac{1}{t^2(\log(t))^{3/2}}
+\frac{3}{2t^2(\log(t))^{5/2}}\right)\nn\\
&\leq&\frac{3}{t(\log(t))^{3/2}}.
\end{eqnarray*}

Since \(t\ge e^{4}\), we have
$
\frac{t+1}{t}
=1+\frac1t
\le
\frac43,
\;
1+\frac{3}{2\log t}
\le
1+\frac38
=
\frac{11}{8}.$
Hence,
$
1+\frac{t+1}{t}
\left(
1+\frac{3}{2\log t}
\right)
\le
1+\frac43\cdot\frac{11}{8}
=
\frac{17}{6}
<
3,
$
which implies
$
|F'(t)|
\le
\frac{17}{6t(\log t)^{3/2}}
<
\frac{3}{t(\log t)^{3/2}}.
$

By the mean-value theorem, there exists
$\zeta_k\in(k+K_0,k+K_0+1)$ such that
\begin{eqnarray*}
 &&|\rho_{k+1}-\rho_k|
 =|F(k+K_0+1)-F(k+K_0)|
 =|F'(\zeta_k)|\nn\\
 &&
 \leq
 \frac{3}{\zeta_k(\log \zeta_k)^{3/2}}\leq\frac{3}{(k+K_0)(\log (k+K_0))^{3/2}}\leq\frac{6}{(k+K_0+1)(\log (k+K_0))^{3/2}}
 \leq 12\alpha_k d_k^3.
\end{eqnarray*}
\noindent Next, since
\[
    u_{k+1}=\cos\beta_k\,u_k+\sin\beta_k\,e_k,
\]
we have
\begin{eqnarray}
    &&v_k=\frac{r_ku_k-r_{k+1}u_{k+1}}{\alpha_k}=A_ku_k-B_ke_k,\nn\\
    &&A_k:=-\rho_k+c_k,
    \;\;
    c_k:=\frac{r_{k+1}(1-\cos\beta_k)}{\alpha_k},
    \;\;
    B_k:=\frac{r_{k+1}\sin\beta_k}{\alpha_k}.
    \label{eq:app-v-frame}
\end{eqnarray}
%
Using \(1-\cos t\leq t^2/2\), \(|\sin t|\leq|t|\), and the definition of
\(\beta_k\), we obtain
\begin{equation}
    0\leq c_k\leq \frac{1}{2}\eta^2\alpha_kd_k^4,
    \qquad
    |B_k|\leq\eta d_k^2.
    \label{eq:app-c-B-size}
\end{equation}
Combining \eqref{eq:app-rho-size}, \eqref{eq:app-v-frame},
\eqref{eq:app-c-B-size}, and \(d_k\leq\eta/2\) and $\eta\alpha_k d_k^2\le 1$, we conclude that
\[
    \norm{v_k}
    \leq |A_k|+|B_k|
    \leq 2 d_k^3+\frac{1}{2}\eta^2 \alpha_k d_k^4 +\eta d_k^2
    \leq 2 \cdot\frac{\eta}{2}\cdot d_k^2+\frac{1}{2}\eta d_k^2+\eta d_k^2\le\frac{5}{2}\eta d_k^2.
\]
This proves \eqref{eq:v-size} by setting $\widehat C_2:=5/2$.

Next, we show (\ref{eq:v-consecutive}).
First, monotonicity of \(\alpha_k\) and \(d_k\) gives
\begin{equation}
    |c_{k+1}-c_k|
    \leq c_{k+1}+c_k
    \leq \eta^2\alpha_kd_k^4.
    \label{eq:app-c-difference}
\end{equation}
 Moreover,
\eqref{eq:d-first} and \eqref{eq:app-d-square} imply
\[
    |r_{k+2}-r_{k+1}|\leq 2 \alpha_{k+1}d_{k+1}^3
\]
and
\begin{eqnarray}\label{addkey}
    \left|
       r_{k+2}d_{k+1}^2-r_{k+1}d_k^2
    \right|
    \leq 6\alpha_kd_k^4.
\end{eqnarray}

We next derive an upper bound on $|B_{k+1}-B_k|$. Recall that
\[
B_k=\eta r_{k+1}d_k^2\sinc(\beta_k),
\qquad
B_{k+1}=\eta r_{k+2}d_{k+1}^2\sinc(\beta_{k+1}),
\]
where
\[
\sinc(t):=
\begin{cases}
\dfrac{\sin t}{t}, & t\neq 0,\\[1mm]
1, & t=0.
\end{cases}
\]
Hence
\begin{align*}
B_{k+1}-B_k
={}&
\eta r_{k+2}d_{k+1}^2\sinc(\beta_{k+1})
-\eta r_{k+1}d_k^2\sinc(\beta_k)\\
={}&
\eta\bigl(r_{k+2}d_{k+1}^2-r_{k+1}d_k^2\bigr)
+\eta r_{k+2}d_{k+1}^2
   \bigl(\sinc(\beta_{k+1})-1\bigr)
-\eta r_{k+1}d_k^2
   \bigl(\sinc(\beta_k)-1\bigr).
\end{align*}
Therefore,
\begin{align*}
|B_{k+1}-B_k|
\le{}&
\eta\left|r_{k+2}d_{k+1}^2-r_{k+1}d_k^2\right|+
\eta r_{k+2}d_{k+1}^2
\left|\sinc(\beta_{k+1})-1\right|+
\eta r_{k+1}d_k^2
\left|\sinc(\beta_k)-1\right|.
\end{align*}
Since \(0\le r_k\le 1\), \(d_{k+1}\le d_k\), and, for \(|t|\le 1\),
$
|\sinc(t)-1|\le \frac{t^2}{6},
$
we obtain
\begin{align*}
&\eta r_{k+2}d_{k+1}^2
\left|\sinc(\beta_{k+1})-1\right|
+
\eta r_{k+1}d_k^2
\left|\sinc(\beta_k)-1\right|\\
&\qquad\le
\frac{\eta}{6}
\left(
d_{k+1}^2\beta_{k+1}^2+d_k^2\beta_k^2
\right)\\
&\qquad\le
\frac{\eta d_k^2}{6}
\left(
\beta_{k+1}^2+\beta_k^2
\right).
\end{align*}
Thus,
\[
|B_{k+1}-B_k|
\le
\eta\left|r_{k+2}d_{k+1}^2-r_{k+1}d_k^2\right|
+
\frac{1}{6}\eta d_k^2
\left(\beta_k^2+\beta_{k+1}^2\right).
\]

Using (\ref{addkey})
together with
\[
\beta_k=\eta\alpha_kd_k^2,
\qquad
\beta_{k+1}\le\beta_k,
\]
we further have
\begin{align*}
\frac{1}{6}\eta d_k^2(\beta_k^2+\beta_{k+1}^2)
&\le
\frac{1}{3}\eta d_k^2\beta_k^2=
\frac{1}{3}\eta^3\alpha_k^2d_k^6=
\frac{1}{3}(\eta^2\alpha_kd_k^2)\,
\eta\alpha_kd_k^4\le
\frac{1}{3}\eta\alpha_kd_k^4,
\end{align*}
where the last inequality uses
\(0<\eta\alpha_kd_k^2\le1\). Consequently,
\[
|B_{k+1}-B_k|
\le
\frac{19}{3}\eta\alpha_kd_k^4.
\]

Equations \eqref{eq:app-rho-size}, \eqref{eq:app-c-B-size} and
\eqref{eq:app-c-difference}, together with the preceding estimate for $|\rho_{k+1}-\rho_k|$, yield
\begin{align}
    |A_k|
    &\leq 2 d_k^3+\frac{\eta^2}{2}\alpha_kd_k^4,
    \label{eq:app-A-size}\\
    |A_{k+1}-A_k|
    &\leq 12\alpha_kd_k^3
          +\eta^2\alpha_kd_k^4.
    \label{eq:app-A-difference}
\end{align}
The orthonormal frames \((u_k,e_k)\) and
\((u_{k+1},e_{k+1})\) differ by a rotation through \(\beta_k\); hence
\[
    \norm{u_{k+1}-u_k}=2\sin(\beta_k/2)\leq\beta_k,
    \qquad
    \norm{e_{k+1}-e_k}=2\sin(\beta_k/2)\leq\beta_k.
\]
Using the representation \eqref{eq:app-v-frame}, we therefore obtain
\begin{align}
    \norm{v_{k+1}-v_k}
    &\leq |A_{k+1}-A_k|+|B_{k+1}-B_k|\notag\\
    &+\bigl(|A_{k+1}|+|B_{k+1}|\bigr)\beta_k\notag\\
    &\leq 12 \alpha_k d_k^3+\left(2\eta^2\alpha_k+\frac{19}{3}\eta\alpha_k\right)d_k^4+2\eta\alpha_k d_k^5+\frac{1}{2}\eta^3\alpha_k^2d_k^6.
    \label{eq:app-v-difference-preliminary}
\end{align}

Indeed, dividing the right-hand side of \eqref{eq:app-v-difference-preliminary} by
$\eta^2\alpha_kd_k^2$ and using
$d_k\leq\eta^2$, $\eta\leq 1/2$, $\alpha_k\leq 1$, and
$\eta\alpha_kd_k^2\leq 1$, we obtain
\[
    12+2\eta^4+\frac{19}{3}\eta^3+2\eta^5
    +\frac12\eta^9
    <\frac{27}{2},
\]
where the last inequality follows from $\eta\leq 1/2$.
Thus
\begin{equation}
    \norm{v_{k+1}-v_k}
    \leq \frac{27}{2}\eta^2\alpha_kd_k^2.
    \label{eq:app-v-difference-final}
\end{equation}
On the other hand, Lemma~\ref{lem:chord-lower-bound} gives
\begin{equation}
    \norm{p^{k+1}-p^k}
    \geq\frac{\beta_k}{\pi}
    =\frac{\eta\alpha_kd_k^2}{\pi}.
    \label{eq:app-step-lower-bound}
\end{equation}
Combining \eqref{eq:app-v-difference-final} and
\eqref{eq:app-step-lower-bound} proves
\[
    \norm{v_{k+1}-v_k}
    \leq \frac{27}{2}\pi\eta\norm{p^{k+1}-p^k},
\]
which is \eqref{eq:v-consecutive} by setting $\widehat C_3:= \frac{27}{2}\pi$.

\medskip
Next, we prove \eqref{eq:a-size}.
Using \eqref{eq:v-size},
\begin{align*}
    a_k=\sum_{\ell=k}^{\infty}\alpha_\ell\norm{v_\ell}^2\leq (5/2)^2\eta^2
      \sum_{\ell=k}^{\infty}
      \frac{1}{2(\ell+K_0+1)\log^2(\ell+K_0)}.
\end{align*}
The integral test gives
\[
    \sum_{\ell=k}^{\infty}
      \frac{1}{(\ell+K_0+1)\log^2(\ell+K_0)}
      \leq\sum_{n=k+K_0}^{\infty}\frac{1}{n\log^2 n}
    \leq\frac{2}{\log(k+K_0)}
    =2 d_k^2.
\]
Therefore
\[
    0\leq a_k\leq \frac{25}{4} \eta^2d_k^2,
\]
which is \eqref{eq:a-size} by setting $\widehat C_4:=\frac{25}{4} $.
\end{proof}

\section{Local half-turn  estimate}\label{app:local-whitney}

\begin{proposition}\label{prop:local-half-turn}
Suppose $i>j$ and
$
0\le \phi:=\varphi_i-\varphi_j\le\pi.
$
Then,
$$
L_{j,i}:=
\sum_{\ell=j}^{i-1}\|p^{\ell+1}-p^\ell\|
\le
\left(1+\frac{\pi}{\sqrt2}\right)\|p^i-p^j\|.
$$
Consequently,
\begin{eqnarray}\label{E1:2}
\|v_i-v_j\|
\le
\widehat C_5\eta\|p^i-p^j\|,
\end{eqnarray}
where $\widehat C_5=\frac{27}{2}\pi\left(1+\frac{\pi}{\sqrt{2}}\right)$.
\end{proposition}

\begin{proof}
Since $r_k$ is increasing,
$
p^{\ell+1}-p^\ell
=
(r_{\ell+1}-r_\ell)u_{\ell+1}
+
r_\ell(u_{\ell+1}-u_\ell).
$
Moreover,
$$
\|u_{\ell+1}-u_\ell\|
=
2\sin\!\left(\frac{\varphi_{\ell+1}-\varphi_\ell}{2}\right)
\le
\varphi_{\ell+1}-\varphi_\ell.
$$
Hence
\[
\|p^{\ell+1}-p^\ell\|
\le
r_{\ell+1}-r_\ell
+
r_i(\varphi_{\ell+1}-\varphi_\ell).
\]
Summing from $\ell=j$ to $i-1$ yields
\[
L_{j,i}
\le
r_i-r_j+r_i\phi.
\]
Define
$
D:=\|p^i-p^j\|.
$
The polar-distance identity gives
$
D^2
=
(r_i-r_j)^2
+
4r_ir_j\sin^2\!\left(\frac{\phi}{2}\right).
$
Thus \begin{eqnarray}\label{add3}r_i-r_j\le D.\end{eqnarray}
 Since
$0\le\phi/2\le\pi/2$,
$
\sin\!\left(\frac{\phi}{2}\right)\ge\frac{\phi}{\pi},
$
and therefore
$
D
\ge
\frac{2}{\pi}\sqrt{r_ir_j}\,\phi.
$
Hence
\[
r_i\phi
\le
\frac{\pi}{2}
\sqrt{\frac{r_i}{r_j}}\,D
\le
\frac{\pi}{\sqrt2}D,
\]
where we used $r_i\le1$ and $r_j\ge1/2$. Consequently,
\begin{eqnarray}\label{add4}
L_{j,i}
\le
\left(1+\frac{\pi}{\sqrt2}\right)D.
\end{eqnarray}

\noindent Finally, summing \eqref{eq:v-consecutive} from $\ell=j$ to $i-1$, and using (\ref{add3}), (\ref{add4}), we have
\[
\|v_i-v_j\|
\le
\sum_{\ell=j}^{i-1}\|v_{\ell+1}-v_\ell\|
\le
\frac{27}{2}\pi\eta
\sum_{\ell=j}^{i-1}\|p^{\ell+1}-p^\ell\|
\le
\frac{27}{2}\pi(1+\frac{\pi}{\sqrt{2}})\eta\|p^i-p^j\|.
\]
Thus, (\ref{E1:2}) holds with $\widehat C_5:=\frac{27}{2}\pi(1+\frac{\pi}{\sqrt{2}})$.
\end{proof}

\begin{proposition}[Local quadratic Whitney remainder]
\label{prop:local-quadratic}
Under the assumptions of Proposition~\ref{prop:local-half-turn},
\[
\left|
a_i-a_j-\langle v_j,p^i-p^j\rangle
\right|
\le
\widehat C_6 \eta\|p^i-p^j\|^2,
\]
where $\widehat C_6:=\frac{27}{2}\pi\left(1+\frac{\pi}{\sqrt{2}}\right)^3$.
\end{proposition}

\begin{proof}
From
\[
p^{\ell+1}-p^\ell=-\alpha_\ell v_\ell
\]
and
\[
a_{\ell+1}-a_\ell=-\alpha_\ell\|v_\ell\|^2,
\]
we obtain
\[
p^i-p^j
=
-\sum_{\ell=j}^{i-1}\alpha_\ell v_\ell,
\qquad
a_i-a_j
=
-\sum_{\ell=j}^{i-1}\alpha_\ell\|v_\ell\|^2.
\]
Consequently,
\[
a_i-a_j-\langle v_j,p^i-p^j\rangle
=
\sum_{\ell=j}^{i-1}
\alpha_\ell
\langle v_j-v_\ell,v_\ell\rangle.
\]
For every $\ell\in\{j,\ldots,i\}$, the angular separation
$\varphi_\ell-\varphi_j$ also lies in $[0,\pi]$. Thus
Proposition~\ref{prop:local-half-turn} yields
\[
\|v_j-v_\ell\|
\le
\widehat C_5\eta\|p^\ell-p^j\|.
\]
Using
\[
\alpha_\ell\|v_\ell\|
=
\|p^{\ell+1}-p^\ell\|,
\]
we have
\[
\left|
a_i-a_j-\langle v_j,p^i-p^j\rangle
\right|
\le
\widehat C_5\eta
\sum_{\ell=j}^{i-1}
\|p^\ell-p^j\|
\,
\|p^{\ell+1}-p^\ell\|.
\]
Since
$
\|p^\ell-p^j\|
\le
L_{j,i},
$
it follows that
$
\left|
a_i-a_j-\langle v_j,p^i-p^j\rangle
\right|
\le
\widehat C_5\eta L_{j,i}^2.
$
Applying Proposition~\ref{prop:local-half-turn} once more proves the claim by setting $\widehat C_6:=\frac{27}{2}\pi\left(1+\frac{\pi}{\sqrt{2}}\right)^3$.
\end{proof}

\section{Different-turn separation}

\begin{lemma}
\label{lem:separation}
If \(i>j\) and \(\varphi_i-\varphi_j\geq\pi\), then
\[
    \norm{p^i-p^j}
    \geq c\,\max\{d_i,d_j\},
    \qquad
    c:=1-e^{-\pi/2}>0.
\]
\end{lemma}

\begin{proof}
For every integer \(n\geq2\), monotonicity of
\(t\mapsto1/(t\log t)\) gives
\[
    \log\log(n+1)-\log\log n
    =
    \int_n^{n+1}\frac{dt}{t\log t}
    \geq
    \frac{1}{(n+1)\log(n+1)}
    \geq
    \frac{1}{2(n+1)\log n}.
\]
Thus, setting \(n=k+K_0\) (where $K_0$ is defined in Lemma \ref{lem:elementary}) in the preceding estimate yields
\begin{align*}
    \varphi_i-\varphi_j
    &=
    \frac{\eta}{2}
    \sum_{k=j}^{i-1}
    \frac{1}{(k+K_0+1)\log(k+K_0)}\\
    &\leq
    \eta
    \log\left(\frac{\log(i+K_0)}{\log(j+K_0)}\right).
\end{align*}
If \(\varphi_i-\varphi_j\geq\pi\), then
\[
    \frac{\log(i+K_0)}{\log(j+K_0)}
    \geq e^{\pi/\eta},
    \qquad
    \frac{d_i}{d_j}
    =
    \sqrt{\frac{\log(j+K_0)}{\log(i+K_0)}}
    \leq e^{-\pi/(2\eta)}
    \leq e^{-\pi/2}.
\]
Since \(d_i\leq d_j\), the reverse triangle inequality yields
$
    \norm{p^i-p^j}
    \geq
    \left|\norm{p^i}-\norm{p^j}\right|
    =
    d_j-d_i
    \geq
    (1-e^{-\pi/2})d_j.$
\end{proof}

\end{document}